\documentclass[12pt]{amsart}
\usepackage{indentfirst}
\usepackage{bm, mathrsfs, graphics,float,amssymb,amsmath,setspace,graphicx,amsthm,epstopdf,subfigure, enumerate, color, bbm}
\usepackage{thmtools, thm-restate}
\usepackage[Symbol]{upgreek}
\usepackage[hidelinks]{hyperref}
\usepackage[margin = 0.8in]{geometry}
\def\R{\mathbb{R}}

\def\P{\mathbb{P}}

\def\p{\partial}
\def\T{\mathbb{T}}
\def\vp{\varphi}

\newcommand{\dd}{\mathrm{d}}

\newcommand{\mb}{\mathbf}

\newcommand{\E}{\mathbb{E}}

\numberwithin{equation}{section}

\newtheorem{lemma}{Lemma}[section]
\newtheorem{theorem}{Theorem}[section]

\newtheorem{definition}{Definition}[section]

\newtheorem{remark}{Remark}[section]

\begin{document}
	\title{Diffusion approximations of Oja's online principal component analysis}

	\author[J.-G. Liu]{Jian-Guo Liu}
	\address{Department of Mathematics and Department of
		Physics, Duke University, Durham, NC}
	\email{jliu@math.duke.edu}
	
	\author[Z. Liu]{Zibu Liu}
	\address{Department of Mathematics, Duke University, Durham, NC}
	\email{zibu.liu@duke.edu}
	
	\maketitle

\begin{abstract}
	Oja's algorithm of principal component analysis (PCA) has been one of the methods utilized in practice to reduce dimension. In this paper, we focus on the convergence property of the discrete algorithm. To realize that, we view the algorithm as a stochastic process on the parameter space and semi-group. We approximate it by SDEs, and prove large time convergence of the SDEs to ensure its performance. This process is completed in three steps. First, the discrete algorithm can be viewed as a semigroup: $S^k\varphi=\E[\varphi(\mb W(k))]$. Second, we construct stochastic differential equations (SDEs) on the Stiefel manifold, i.e. the diffusion approximation, to approximate the semigroup. By proving the weak convergence, we verify that the algorithm is 'close to' the SDEs. Finally, we use reversibility of the SDEs to prove long time convergence.   
\end{abstract}
\section{Introduction} \label{sec:notations}
Principal component anlysis (PCA) is a basic tool in dimension reduction. Due to explosion of data, command of efficient PCA algorithms is increasing. In this paper, we focus on the online PCA algorithm proposed by Oja in \cite{oja1985stochastic}, which is also named as the stochastic gradient ascent (SGA) method.

Suppose that $\mb x\in \R^n$ is a mean zero random variable (R.V.). Let 
\begin{align}
	\mb A :=\E[\mb x\mb x^T]
\end{align} 
be the covariance matrix. Traditional PCA algorithms diagonalize $\mb A$ to derive principal eigenvectors (i.e. the principal components) of $\mb A$. However, due to limitation of storage and high dimension of data in recent fields such as deep learning, explicit form of the dense matrix $\mb A$ may not be available. Therefore, practitioners prefer 'online' algorithms: it only requires a limited amount of samples of $\mb x$ in each iteration. To solve this problem, Oja proposed the following SGA method in \cite{oja1985stochastic}: 

\begin{equation} \label{eq:SGA}
\begin{aligned} 
\mb{w_j}(k) &= \mb{w_j}(k-1) + \eta(k)\mb{x}^T(k)\mb{w_j}(k-1)[\mb{x}(k)-(\mb{x}^T(k)\mb{w_j}(k-1))\mb{w_j}(k-1)\\
&-2\sum_{i=1}^{j-1}(\mb{x}^T(k)\mb{w_i}(k-1))\mb{w_i}(k-1)],\ j=1,\ 2,\ ...,\ p.
\end{aligned}
\end{equation}
This algorithm iterates the first $p$ principal components $\mb{w_j}\in \R^n, j=1,2,...,p$. Here $\mb x(k),\ k=1,2,...$ are independent samples of $\mb x$, $\eta(k),\ k=1,2,...$ are learning rates.

Algorithm \eqref{eq:SGA} determines a discrete time Markovian process, i.e. $\mb W(k)=[\mb{w_1}(k),\mb{w_2}(k),...,\mb{w_p}(k)]$. The main goal of this paper is to gain a good understanding of this random process (R.P.) from the view of semigroups, diffusion approximations and SDEs. 

First of all, as $\eta\to 0$, replacing $\mb x\mb x^T$ by $\mb A$ in \eqref{eq:SGA}, we derive the corresponding ODE:
\begin{gather}\label{eq:ndODE}
\left\{
\begin{split}
\mb{\dot{q_1}}&=\mb{Aq_1}-\mb{(q_1\cdot Aq_1)q_1},\\
\mb{\dot{q_j}}&=\mb{Aq_j}-\mb{(q_j\cdot Aq_j)q_j}-2\sum_{i=1}^{j-1}\mb{(q_i\cdot Aq_j)q_i},\ j=2,\ 3,\ ...,\ n.\\
\mb{q_i}(0)&=\mb{q_{i,0}},\ i=1,\ 2,\ ...,\ p.
\end{split}\right.
\end{gather}
Convergence properties including global convergence, stable manifolds and exponential convergence were thoroughly investigated in our previous work \cite{liu2022convergence}. In particular, we proved that for almost every initial value $\mb{Q_0}\in O(n)$, the solution exponentially converges to the eigenbasis (up to a sign). Moreover, the eigenvectors are aligned in a descending order of the eigenvalues. See Theorem 5.2 in \cite{liu2022convergence}. As far as we know, this is the first complete result providing global exponential convergence and closed formula for stable manifolds of a PCA flow \cite{blondel2004unsolved}.

Given convergence of the corresponding ODE, we aim at proving similar result for the discrete algorithm \eqref{eq:SGA} in this paper. We consider this problem in three steps: Viewing the SGA iteration as a semigroup, we construct proper diffusion approximations and prove convergence of diffusion approximations to ensure the performance of the algorithm. 

First, we view the SGA method as a semigroup. It can be reformulated in the following form:
\begin{align*}
	\mb W(k+1) = \mb W(k) + \eta\cdot\mb G(\mb x(k+1)\mb x^T(k+1),\mb W(k)).
\end{align*}
Here $\mb G\in \R^{n\times p}$ is defined in \eqref{def:G}. For arbitrary test function $\varphi\in C(\R^{n\times p})$, define
\begin{align}
	S\varphi(\mb W) := \E\varphi(\mb W+\eta\mb G(\mb x\mb x^T, \mb W)). 
\end{align}
Under this notation, if the initial datum of the SGA method is $\mb{W_0}$, then the Markovian property yields
\begin{align}
	 S^k\varphi(\mb W_0)=\mb \E\vp(\mb W(k)).
\end{align} 
Thus, convergence of the SGA method can also be interpreted as the convergence of the semigroup $\{S^k\}, k=1,2,...$.

Second, we construct appropriate diffusion approximations. Although the SGA method does not preserve $\mb W(k)$ to stay on the Stiefel manifold $O(n\times p)$, the desired result, i.e. the eigenbasis, is in $O(n\times p)$. Thus, we aim at deriving a good diffusion approximation of the semigroup $S$ and the SGA method. It should be an SDE that stays on the Stiefel manifold. The classical method to derive an SDE on a certain manifold is to project a Stratanovich SDE onto the desired manifold \cite{hsu2002stochastic}:
\begin{align*}
	\dd \mb W = \mathcal{P}_{T_\mb W\mathcal{M}}(\mb G(\mb A,\mb W)\dd t+\bm{\sigma}\circ\dd \mb B).
\end{align*}
Here $\mathcal{P}_{T_\mb Z\mathcal{M}}$ is the projection operator onto the tangent space at $\mb W$ on $\mathcal{M}$. If the semigroup is close to the diffusion process, then by proving convergence of the diffusion process in some sense, we can also guarantee the performance of the algorithm. 

Finally, we prove convergence of the diffusion process. The way to prove it is by seeking 'reversibility'. In fact, if the Fokker-Planck equation of the SDE can be recast in the following form:
\begin{align*}
	\p_t\rho = \nabla\cdot(\rho\nabla U+\nabla \rho) = \nabla\cdot \left(e^{-U}\nabla\left(\dfrac{\rho}{e^U}\right)\right)
\end{align*}
for some potential $U$, then the diffusion process satisfies detailed-balance condition, i.e., the process is reversible. Then, Poincare's inequality can ensure the exponential convergence of $\rho$ in a certain $L^2$ sense. This proves the convergence of SDEs, which also finishes our analysis. 

Under this framework of analysis, we will provide our main results and revise previous literature.

\subsection{Previous results and unsolves problems}
One of the important features of \eqref{eq:SGA} which other algorithms do not possess is its semi-decoupling feature: iteration of $\mb{w_j}$ does not depend on $\mb{w_i},\ i>j$. This feature facilitates its implementation in neural networks \cite{oja1992principal}, thus researchers focus on it. This feature was also extended to the corresponding ODE, i.e. \eqref{eq:finvODE}. Based on this semi-decoupling property, we also proved all convergence results of \eqref{eq:finvODE} in \cite{liu2022convergence}. 

However, a satisfying convergence result for \eqref{eq:SGA} is still wanting. Since Oja and Karhunen proposed \eqref{eq:SGA} in \cite{oja1985stochastic}, its convergence behavior has always been a focus in analysis of online PCA. Oja and Karhunen used stochastic approximation to derive almost sure convergence of \eqref{eq:SGA} under an implicit condition on the distribution of $\mb{x}$ \cite{oja1985stochastic,oja1982simplified,oja1992principal}. This implicit condition requires the iteration to visit a compact set containing the equilibrium for infinitely many times. However, this condition is difficult to verify in practice.

To improve Oja's result, more recently, authors in \cite{li2017diffusion} derived weak convergence of the first component of \eqref{eq:SGA} to a multidimensional Ornstein-Uhlenbeck process. Following \cite{li2017diffusion}, the algorithm conducting full orthonormalization was considered and the weak convergence of all components was derived \cite{liang2017nearly}. 

The diffusion approximation of the first component of \eqref{eq:SGA} was also considered in our previous work \cite{feng2018semigroups} in which both first -order and second-order approximation were derived. As a corollary, the weak convergence of the first component of the SGA method was verified. 
However, a diffusion approximation of the whole SGA iteration method \eqref{eq:SGA} is still an open problem.

The main tool utilized in \cite{feng2018semigroups} is the Lax equivalence theorem. An alternative stochastic analysis approach to prove the convergence is developed by Milstein \cite{milstein1994numerical}, which was adopted to derive diffusion approximations of the stochastic gradient descent (SGD) method \cite{li2017stochastic}.
\subsection{Main results}
First of all, we investigated properties of the semigroup $S^k,\ k=1,2...,$. In particular, we proved the stability and the regularity of it. For the stability, we proved that for a fixed terminal time $T$, if $\|\mb{W_0}\|_F\leq r$, then there exist constants $C$ and $\eta_0$ that depend on $r,T$ and the distribution of $\mb x$ such that 
\begin{align*}
	\|\mb{W}(k)\|_F\leq C
\end{align*}
holds for all   $k=1,2,...,\left[\dfrac{T}{\eta}\right]$. Here $\eta\in(0,\eta_0)$ is the learning rate in \eqref{eq:SGA}. See Lemma \ref{lmm:stability}. We proved the stability because it is necessary for the application of the Lax equivalence theorem. For the regularity, we prove that $S^k\varphi,\ k=1,2,...$ admit the same order regularity as $\varphi$, i.e.
\begin{align*}
	\|S^k\varphi\|_{C^m(B(\mb 0,r))}\leq C\|\varphi\|_{C^m(B(\mb 0,r'))}.
\end{align*}
Here $C$ and $r'$ are constants that depend on $m,r,T$ and the distribution of $\mb x$. See Theorem \ref{thm:semigroup} for details. 

Second, we constructed the desired diffusion approximations. We proved that the following family of SDEs 
\begin{align*}
\dot{\mb W}=\mb G(\mb A,\mb W)+\eta \mathcal{P}_{T_{\mb W}O(n)}\mb F(\mb W)+\sqrt{\eta}\mathcal{P}_{T_{\mb W}O(n)}\dot{\mb Z},\ \mb W(0)=\mb{W_0}\in O(n)
\end{align*}
will stay on the Stiefel manifold for all $t>0$. See Lemma \ref{lmm:orth}. Here
\begin{align*}
	\dot{\mb{Z}}=(\dot{Z}_{ij})_{n\times n},\  \dot{Z}_{ij}(\mb W)=H_{ijkl}(\mb W)\circ \dot{B}_{kl},
\end{align*}  
where $\mb H=(H_{ijkl})_{n\times n\times n\times n}$ are coefficients and $\dot{B}_{kl}$ is the white noise. See \eqref{eq:diffusionapprox} for detail. 

In fact, the SDE \eqref{eq:diffusionapprox} serves as the first-order diffusion approximation of the SGA method. we proved that under proper regularity conditions of the test function $\varphi$, there exists a constant $C_1=C_1(\mb x,T,\eta)$ such that
\begin{align*}
	\sup\limits_{\mb{W}\in O(n),k\eta\leq T}|S^k\varphi(\mb{W})-u(\mb{W},k\eta)|\leq C_1(\mb x,T,\vp)\eta.
\end{align*}
Here $u(\mb W,t)$ is the solution to the Kolmogorov equation determined by \eqref{eq:diffusionapprox}, with the initial value $\varphi$. See Theorem \ref{thm:diffusionapprox}. The main idea of the proof comes from the Lax equivalence theorem \cite{lax1956survey}: stability and consistence is equivalent to convergence. The consistence is ensured by Taylor's expansion, see Section \ref{sec:appendix} for details. 

A natural question is that whether higher- order approximation exists. Unfortunately, the answer is no. We proved that the possible second order approximation, which is an SDE, does not stay on the Stiefel manifold. See Lemma \ref{lmm:unstable}. This instability is probably due to the omitted higher-order terms in the SGA algorithm: second and higher-order (w.r.t. $\eta$) terms were neglected in \eqref{eq:SGA} when conducting the Gram-Schmidt orthogonalization.

Finally, for two special cases, we proved the exponential convergence of the SDE. As we introduced before, we seek for reversibility to prove the exponential convergence. 

First, we consider the overdamped Langevin equation on the Stiefel manifold:
\begin{align*} 
\dd\mb{Q}(t) = \mathcal{P}_{T_{\mb{Q}}O(n)}\circ(-\nabla U(\mb{Q})\dd t+\sigma\dd\mb{W}(t)).
\end{align*} 
The exponential convergence of it is proved in Section \ref{sec:diff}. If we select the potential $U$ as the weighted Rayleigh quotient (see \cite{liu2022convergence}) and let $\sigma =0$, then the Oja-Brockett flow \cite{brockett1991dynamical} is recovered. We have to emphasize that the overdamped Langevin equation is not a special case of \eqref{eq:diffusionapprox} since $\mb G$ in \eqref{eq:SGA} is not a gradient of a certain potential. 
  
Second, for $n=2$ of \eqref{eq:diffusionapprox}, the SDE is rewritten as
\begin{align*}
\mathrm{d}\mb{W} = \mb{F_1}(\mb{W})\mathrm{d}t +\sqrt{\eta}\cdot c(\mb{W})\mb{W}\circ\mathrm{d}\mb{Z}.
\end{align*} 
Here $\mb{F_1}$ is defined in \eqref{eq:finvODE} and $c(\mb W)$ is a scalar. See details in \eqref{eq:n=2}. Exponential convergence of this case is proved in Theorem \ref{thm:expconvoftheta}. The main approach is to consider the dynamics of the rotational angle of $\mb W\in O(2)$, which is a one-dimensional SDE, and the reversibility automatically holds. 

\section{Premier}
First, we rewrite \eqref{eq:SGA} by matrices. For $\mb{\Lambda},\mb{Q}\in\R^{n\times n}$, define
\begin{equation}\label{def:G}
\begin{aligned}
\mb{\Sigma}(\mb{\Lambda},\mb{Q})&:=\sum_{j=1}^n\sum_{k=1}^{j-1}\mb{E_j}\mb{Q}^T\mb{\Lambda QE_k} - \mb{E_k}\mb{Q}^T\mb{\Lambda QE_j},\\ 
\mb{G}(\mb{\Lambda},\mb{Q})&:=\mb{\Lambda Q}-\mb{QQ}^T\mb{\Lambda}\mb{Q}+\mb{Q}\mb{\Sigma}(\mb{\Lambda},\mb{Q}).
\end{aligned}
\end{equation}
Then \eqref{eq:SGA} also reads as
\begin{gather} \label{eq:rewriteleading}
\left\{
\begin{split}
\mb{A}(k)&=\mb{x}(k)\mb{x}^T(k),\\
\mb{W}(k)&=\mb{W}(k-1)+\eta_k\mb{G}(\mb{A}(k),\mb{W}(k-1)). 
\end{split}\right.
\end{gather}

In our previous paper \cite{liu2022convergence}, we thoroughly investigated corresponding ODE, which can be written as:
\begin{gather}\label{eq:smpODE}
\left\{
\begin{split}
\mb{\dot{Q}} &=  \mb{Q}\sum_{j=1}^n\sum_{k=1}^{j-1}(\mb{E_j}\mb{Q}^T\mb{AQE_k} - \mb{E_k}\mb{Q}^T\mb{AQE_j}),\\
\mb{Q}(0) &= \mb{Q_0}\in O(n).
\end{split}\right.
\end{gather}
We define
\begin{align}\label{def:F}
\mb{F_1(Q)}:=\mb{Q}\mb{\Sigma}(\mb{A},\mb{Q}). 
\end{align}
Thus one can rewrite \eqref{eq:smpODE} as
\begin{gather}\label{eq:finvODE}
\left\{
\begin{split}
\mb{\dot{Q}} &= \mb{F_1(Q)},\\
\mb{Q}(0) &= \mb{Q_0}\in O(n).
\end{split}\right.
\end{gather}
From now on, we will use \eqref{eq:finvODE} in all proofs.
\subsection{Notations and assumptions}
We will follow the convention of notations in our previous paper \cite{liu2022convergence}.

We assume that $\nu$ is compact, i.e., there exists a constant $M>0$ such that
\begin{align} \label{ass:L_inf}
\|\mb{x}\|_2\leq M.
\end{align}

In the following sections, we will adopt both the matrix representation and the component-wise representation of \eqref{eq:finvODE}, thus we clarify the notation here. $\mb{q_i},\ i=1,2,...,n$ represent the column vectors of $\mb{Q}$ in order, i.e.
\begin{align}
\mb{Q} = [\mb{q_1},\ \mb{q_2},\ ...,\  \mb{q_n}],
\end{align}  
while $\tilde{\mb{q_i}},\ i=1,2,...,n$ represent the row vectors of $\mb{Q}$ in order, i.e.
\begin{align}
\mb{Q}^T = [\tilde{\mb{q_1}}^T,\ \tilde{\mb{q_2}}^T,\ ...,\ \tilde{\mb{q_n}}^T].
\end{align}
For each entry, $q_{i,j},\ i,j=1,2,...,n$ represent the entries at $i$th row, $j$th column of the matrix $\mb{Q}$, i.e.
\begin{align}
\mb{q_j} = (q_{1,j},\ q_{2,j},\ ...,\  q_{n,j})^T.
\end{align} 
The canonical orthonormal basis in $\R^n$ is denoted as $\mb{e_j}, j=1,2,...,n$, which are written in column vectors, i.e.
\begin{align}
\mb{I_n} = [\mb{e_1},\ \mb{e_2},\ ...,\ \mb{e_n}].
\end{align}
Here $\mb{I_n}$ is the identity matrix of size $n$.

For $\mb{M},\mb{N}\in\R^{n\times n}$, $\|\mb{M}\|_{F}$ represents the Frobenius norm of $\mb{M}$ and $\langle \mb{M},\ \mb{N}\rangle_{F}$ represents the inner product in the Frobenius sense:
\begin{align}
\|\mb{M}\|=\sqrt{\mathrm{tr}(\mb{M}\mb{M}^T)},\ \langle \mb{M},\ \mb{N}\rangle_{F}=\mathrm{tr}(\mb{M}^T\mb{N}).
\end{align}
For $\mb{x}=(x_1,x_2,...,x_n)\in\R^n$,  $\|x\|_2$ represents the $\ell_2$ norm of $\mb{x}$, i.e.
\begin{align}
\|\mb{x}\|_{2}=\sqrt{\sum_{j=1}^n|x_j|^2}
\end{align}

Suppose that the eigenvalues of $\mb{A}$ are all single, i.e. of multiplicity one. Denote them as
\begin{align}\label{eq:Aeigenvalue} \lambda_1>\lambda_2>...>\lambda_n>0
\end{align}
in descending order. Without loss of generality, we assume that $\mb A$ is diagonal:
\begin{gather*}
\mb A=\mathrm{diag}\{\lambda_1,\lambda_2,...,\lambda_n\}.
\end{gather*}

By default, omitted proofs of Lemmas and other important but complicated computations are available Section \ref{sec:appendix}.

\section{Diffusion approximation of the online PCA algorithm} \label{sec:diff}
In this section, we consider the iteration scheme \eqref{eq:SGA}. 
We also assume that the learning rates are constant, i.e. $\eta_k=\eta>0, k=1,2,...$. Under these assumptions, we derived the diffusion approximation of \eqref{eq:SGA}. Our main results imply that \eqref{eq:finvODE} is the weak limit of \eqref{eq:SGA} as $\eta$ approaches 0. We also derived families of matrix-valued SDEs (invariant in the Steifel manifold) which serve as first order weak approximations. In particular, \eqref{eq:finvODE} is understood as a special case of these SDEs by taking the time step size $\eta=0$. 

\subsection{The semigroup}
The matrix-valued discrete time Markov process defined in \eqref{eq:SGA}, i.e. $\mb{W}(k), k=1,\ 2,\ ...$, is time homogeneous because $\mb{x}(k)$ share the same distribution.

Following the notations in \cite{feng2018semigroups}, we denote the expectation under the distribution of this Markov chain starting from $\mb{W_0}$ as $\E_{\mb{W_0}}$. In our discussion, $\mb{W_0}$ is assumed to be deterministic though it could be a random variable in general contexts. Denote the law of $\mb{W}(k)$ (starting from $\mb{W_0}$) as $\mu^k(\cdot;\mb{W_0})$ and the transition probability as $\mu(\mb{V},\cdot)$. Then by the Markov property, for any Borel set $E\subset \R^{n\times n}$,
\begin{align*}
\mu^{k+1}(E;\mb{W_0})=\int_{\R^{n\times n}}\mu(\mb{V},E)\mu^k(\mathrm{d}\mb{V};\mb{W_0})=\int_{\R^{n\times n}}\mu^k(E;\mb{U})\mu(\mb{W_0},\mathrm{d}\mb{U}).
\end{align*}  
For a fixed test function $\vp\in L^\infty(\R^{n\times n})$, define
\begin{align} \label{def:u}
	u^k(\mb{W_0}) = \E_{\mb{W_0}}\left[\vp(\mb{W}(k))\right]=\int_{\R^{n\times n}}\vp(\mb{V})\mu^k(\mathrm{d}\mb{V};\mb{W_0}),\ k=0,1,2,...
\end{align}
Here $\mb{W}(k)$ is defined in \eqref{eq:SGA}. The Markov property yields
\begin{align*}
	u^{k+1}(\mb{W_0}) &= \E_{\mb{W_0}}[\E_{\mb{W_0}}[\vp(\mb{W}(k+1))|\mb{W}(1)]] \\
	&= \E_{\mb{W_0}}[u^k(\mb{W}(1))] \\
	&= \int_{\R^{n\times n}}\mu(\mathrm{d}\mb{W_1},\mb{W_0})\int_{\R^{n\times n}}\vp(\mb{V})\mu^k(\mathrm{d}\mb{V};\mb{W_1}).
\end{align*}

Then by \eqref{eq:SGA}, we derive that for any $\mb{W}\in\R^{n\times n}$, 
\begin{align} \label{def:semigroup}
	u^{k+1}(\mb{W})=\E u^k(\mb{W}+\eta\mb{G}(\mb{x}\mb{x}^T,\mb{W})):=Su^k(\mb{W}),
\end{align}
hence $u^0(\mb{W})=\vp(\mb{W})$ and $\{S^k\}_{k\geq 0}$ forms a semigroup.

Before discussing the diffusion approximation, we derive some basic properties of the Markov chain and the semigroup. 

\begin{lemma}(stability) \label{lmm:stability}
	Fix a real number $r>0$ and a terminal time $T>0$. Let $\mb{W}(k), k=0,1,2,...$ be the Markov chain generated by \eqref{eq:SGA} with an initial datum $\mb{W_0}$ satisfying $\|\mb{W_0}\|_F\leq r$. Then there exist constants $\eta(r,M,T)>0$ and $C(r,M,T)>0$ which depend on $r$, $T$ and $M$ in \eqref{ass:L_inf} such that for any $0< \eta\leq\eta(r,M,T)$ and $k=0,1,2,...,\left[\dfrac{T}{\eta}\right]$,
	\begin{align}
		\P\left(\|\mb{W}(k)\|_{F}^2\leq C(r,M,T)\right)=1,
	\end{align}
	i.e., $\mb{W}(k)$ is uniformly bounded for any time discretization with time step size less than $\eta(r,M,T)$.
\end{lemma}
See Section \ref{sec:appendix} for the proof of this lemma. 
Based on Lemma \ref{lmm:stability}, we prove that $u^k(\mb{W})$ possesses the same regularity as the test function $\vp$. Admissible sets of test functions are
\begin{align}
C_b^m(\R^{n\times n}):=\left\{f\in C^m(\R^{n\times n})\ \big| \|f\|_{C^m(\R^{n\times n})}:=\sum_{|\alpha|\leq m}|D^{\alpha}f|_{\infty}\right\},\ m=1,2,....
\end{align} 
We denote the open ball in $(\R^{n\times n},\|\cdot\|_F)$ centered at $\mb{M}\in\R^{n\times n}$ with radius $r$ by $B(\mb{M},r)$, i.e.
\begin{align} \label{def:ball}
	B(\mb{M},r):=\{\mb{N}\in\R^{n\times n}: \|\mb{N-M}\|_F< r\}.
\end{align}  

\begin{theorem} \label{thm:semigroup}
	(properties of the semigroup) Fix a real number $r>0$ and a terminal time $T>0$. Let $\{u^k\}_{k\geq 0},\ \{S^k\}_{k\geq 0}$ be the semigroup defined in \eqref{def:semigroup} and $M$ be the upper bound in \eqref{ass:L_inf}. Suppose that $\vp\in C_b^m(\R^{n\times n})$ where $m\geq 1$. Then:
	\begin{enumerate}
		\item ($L^{\infty}$ contraction) $S:L^{\infty}(\R^{n\times n})\to L^{\infty}(\R^{n\times n})$ is a contraction;
		\item (regularity) there exist constants $\eta(r,m,M,T)>0$ and $C(r,m,M,T)>0$ such that for any $0<\eta\leq \eta(r,m,M,T)$:
		\begin{align}
		\|u^k\|_{C^m(B(\mb{0},r))}\leq  C(r,m,M,T)\|\vp\|_{C^m(B(\mb{0},e^{T(2M^2+1)/2}r))}.
		\end{align}
	\end{enumerate} 
\end{theorem}
\begin{proof}	
	By Lemma \ref{lmm:stability}, for any $\eta\leq \eta(r,M,T)$, there exists $C(r,M,T)>0$ such that $\|\mb{W}\|_F\leq C(r,M,T)$ holds almost surely for $k=0,1,2,...,[T/\eta]$. Notice that $\mb{G}$ is a polynomial w.r.t. $\mb{W}$, so there exists a constant $C'(r,M,T)>0$ such that for any $\mb{W}\in B(\mb{0},r)$ and indices $(i,j)$ and $(i',j')$,
	\begin{align} \label{eq:boundG}
	\left|\dfrac{\p(\mb{G}(\mb{x}(k)\mb{x}(k)^T,\mb{W}))_{i',j'}}{\p w_{i,j}}\right|\leq C'(r,M,T).
	\end{align}
	Here $(\mb{G})_{i',j'}$ represents the entry at the $i'$th row and $j'$th column in the matrix $\mb{G}$.  
	
	According to the proof of Lemma \ref{lmm:stability}, we know that 
	\begin{align*}
	\|\mb{W}(k)\|_F^2 \leq (1+(2M^2+1)\eta)\|\mb{W}(k-1)\|_F^2, \ k=1,2,...,[T/\eta].
	\end{align*}
	Thus if $\mb{W}\in B(\mb{0},r)$, then 
	\begin{align} \label{eq:bounditer}
	\mb{W}+\eta \mb{G}(\mb{x}(k)\mb{x}(k)^T,\mb{W})\in B(\mb{0},r\sqrt{1+(2M^2+1)\eta})\subset B\left(\mb{0},\left(1+\dfrac{(2M^2+1)\eta}{2}\right)r\right).
	\end{align}
	Denote $C''=(2M^2+1)/2$. Now we proceed to prove the theorem.
	
	(1) Suppose that $\psi\in L^{\infty}(\R^{n\times n})$, then the $L^{\infty}$ contraction is derived directly by \eqref{def:semigroup}:
	\begin{equation}
		\begin{aligned}
		\|S\psi\|_{L^{\infty}(\R^{n\times n})} &= \sup\limits_{\mb{W}\in\R^{n\times n}}|\E\psi(\mb{W}+\eta\mb{G}(\mb{x}\mb{x}^T,\mb{W}))|\\
		&\leq \sup\limits_{\mb{W}\in\R^{n\times n}}|\psi(\mb{W})|\\
		&\leq \|\psi\|_{L^{\infty}(\R^{n\times n})}.
		\end{aligned}
	\end{equation}
	
	(2)
	We prove the case $m=1$ by induction. The proof for the case of $m\geq 1$ is similar. Because $u^0=\vp$, so the conclusion holds for $k=0$. Now for $k\geq 0$, by the dominant convergence theorem (DCT), we have for any $\mb{W}\in B(\mb{0},r)$ and $1\leq i,j\leq n$,
	\begin{equation} \label{eq:chainrule}
		\begin{aligned} 
		\dfrac{\p u^{k+1}}{\p w_{i,j}}(\mb{W}) &= \sum_{1\leq i',j'\leq n}\E\left[\dfrac{\p u^{k}}{\p w_{i',j'}}(\mb{W}+\eta \mb{G}(\mb{x}(k)\mb{x}(k)^T,\mb{W}))\cdot \dfrac{\p (\mb{W}+\eta \mb{G}(\mb{x}(k)\mb{x}(k)^T,\mb{W}))_{i',j'}}{\p w_{i,j}}\right]\\
		&= \E\left[\dfrac{\p u^{k}}{\p w_{i,j}}(\mb{W}+\eta \mb{G}(\mb{x}(k)\mb{x}(k)^T,\mb{W}))\right]\\
		&+\eta\sum_{1\leq i',j'\leq n}\E\left[\dfrac{\p u^{k}}{\p w_{i',j'}}(\mb{W}+\eta \mb{G}(\mb{x}(k)\mb{x}(k)^T,\mb{W}))\cdot \dfrac{\p(\mb{G}(\mb{x}(k)\mb{x}(k)^T,\mb{W}))_{i',j'}}{\p w_{i,j}}\right].
		\end{aligned}
	\end{equation}
 	Remember that \eqref{eq:bounditer} holds, so
 	\begin{equation*}
 		\left| \E\left[\dfrac{\p u^{k}}{\p w_{i,j}}(\mb{W}+\eta \mb{G}(\mb{x}(k)\mb{x}(k)^T,\mb{W}))\right]\right|\leq \left\|\dfrac{\p u^{k}}{\p w_{i,j}}\right\|_{L^\infty(B(\mb 0,(1+C''\eta)r))}
 	\end{equation*} 
 	Substituting this into \eqref{eq:chainrule} yields
	\begin{equation*}
		\begin{aligned}
		\left|\dfrac{\p u^{k+1}}{\p w_{i,j}}(\mb{W})\right|&\leq \left\|\dfrac{\p u^{k}}{\p w_{i,j}}\right\|_{L^\infty(B(\mb 0,(1+C''\eta)r))}\\
		&+\eta C'(r,M,T)\sum_{1\leq i',j'\leq n}\left|\dfrac{\p u^{k}}{\p w_{i',j'}}(\mb{W}+\eta \mb{G}(\mb{x}(k)\mb{x}(k)^T,\mb{W}))\right|.
		\end{aligned}
	\end{equation*}
	Then using \eqref{eq:boundG}, taking $L^{\infty}$ norm on both sides and summing up over $1\leq i,j\leq n$, we derive
	\begin{equation*}
	\begin{aligned}
	\sum_{1\leq i,j\leq n}\left\|\dfrac{\p u^{k+1}}{\p w_{i,j}}\right\|_{L^{\infty}(B(\mb{0},r))}&\leq 
	\sum_{1\leq i,j\leq n}\left\|\dfrac{\p u^{k}}{\p w_{i,j}}\right\|_{L^{\infty}(B(\mb{0},(1+C''\eta)r))}\\
	&+n^2C'(r,M,T)\eta\|u^k\|_{C^1(B(\mb{0},(1+C''\eta)r))}.
	\end{aligned}
	\end{equation*}
	This inequality yields
	\begin{align}
		\|u^{k+1}\|_{C^1(B(\mb{0},r))}\leq (1+n^2C'(r,M,T)\eta)\|u^k\|_{C^1(B(\mb{0},(1+C''\eta)r))}.
	\end{align}
	By induction, we have
	\begin{equation}
		\begin{aligned}
		\|u^{k}\|_{C^1(B(\mb{0},r))}&\leq (1+n^2C'(r,M,T)\eta)^k\|u^0\|_{C^1(B(\mb{0},(1+C''\eta)^kr))}\\
		&\leq (1+n^2C'(r,M,T)\eta)^{T/\eta}\|\vp\|_{C^1(B(\mb{0},(1+C''\eta)^{T/\eta}r))}\\
		&\leq e^{n^2C'(r,M,T)T}\|\vp\|_{C^1(B(\mb{0},e^{T(2M^2+1)/2}r))}.
		\end{aligned}
	\end{equation}
\end{proof}

\subsection{The diffusion approximation}
In this section, we discuss the diffusion approximation of the semigroup \eqref{def:semigroup}. 

\subsubsection{SDEs on the Stiefel manifold}
Because the SGA method aims to derive the correct unit eigenbasis, so we expect that our SDE should stay on the Stiefel manifold. Therefore, we consider the following family of SDEs:
\begin{align}\label{eq:diffusionapprox}
	\dot{\mb W}=\mb G(\mb A,\mb W)+\eta \mathcal{P}_{T_{\mb W}O(n)}\mb F(\mb W)+\sqrt{\eta}\mathcal{P}_{T_{\mb W}O(n)}\dot{\mb Z},\ \mb W(0)=\mb{W_0}\in O(n).
\end{align}
Here $\mathcal{P}_{T_{\mb W}O(n)}$ is the projection onto the tangent space of $O(n)$ at $\mb W$, see \eqref{def:P}; $\dot{\mb Z}$ is defined as
\begin{align}
	\dot{\mb Z}(\mb W)=(\dot{Z}_{ij})_{n\times n},\ \dot{Z}_{ij}(\mb W)=H_{ijkl}(\mb W)\circ\dot{B}_{kl};
\end{align}
$\mb H(\mb W)=(H_{ijkl}(\mb W))_{n\times n\times n\times n}$ is the coefficient tensor of the Brownian motion; $\mb B=(B_{ij})_{n\times n}$ is the standard Brownian motion in $\R^{n\times n}$; The notation '$\circ$' represents that \eqref{eq:diffusionapprox} is an SDE in the Stratonovich sense.

By letting $\eta=0$ in \eqref{eq:diffusionapprox}, the SDE degenerates to the ODE
\begin{align*}
	\dot{\mb W}=\mb G(\mb A,\mb W),\  \mb{W}(0)=\mb{W_0}\in O(n).
\end{align*}
This is exactly \eqref{eq:finvODE}. Thus we expect that \eqref{eq:diffusionapprox} serves as the diffusion approximation of the SGA method.

We first check that for arbitrary $\mb F(\mb W)\in\R^{n\times n}$ and $\mb H(\mb W)\in \R^{n\times n\times n\times n}$, \eqref{eq:diffusionapprox} stays on the Stiefel manifold if $\mb{W_0}\in O(n)$. As preparation, we first rewrite the projection operator.

\begin{lemma} \label{lmm:projection}
	Let $\mathcal{P}_{T_{\mb W}O(n)}$ defined as in \eqref{def:P} for some $\mb W=(w_{ij})_{n\times n}\in O(n)$. Let
	\begin{align}\label{def:tensorP}
		\mb P:=(P_{ijkl})_{n\times n\times n\times n},\ P_{ijkl}:=\dfrac{1}{2}(w_{is}w_{ks}\delta_{jl}-w_{kj}w_{il}).
	\end{align}
	Then we have:
	\begin{enumerate}[(i)]
		\item (projection) for any $\mb F=(f_{ij})_{n\times n}$,
		\begin{align}
			(\mathcal{P}_{T_{\mb W}O(n)}\mb F)_{ij}=P_{ijkl}f_{kl}=\dfrac{1}{2}(f_{ij}-w_{il}f_{kl}w_{kj});
		\end{align}
		\item (symmetry) $P_{ijkl}=P_{klij}$, i.e. $\mb P$ is symmetric;
		\item (idempotence) $P_{ijkl}P_{klrs}=P_{ijrs}$, i.e. $\mb P^2=\mb P$. 
 	\end{enumerate}
\end{lemma}
\begin{proof}
	Because $\mb W\in O(n)$, so $w_{is}w_{ks}=\delta_{ik}$ and 
	\begin{align}\label{eq:P1}
		P_{ijkl}=\dfrac{1}{2}(\delta_{ik}\delta_{jl}-w_{kj}w_{il}).
	\end{align}
	We first prove (i). Because $\mathcal{P}_{T_{\mb W}O(n)}\mb F=\dfrac{1}{2}(\mb F-\mb W\mb F^T\mb W)$, we have
	\begin{align*}
		(\mathcal{P}_{T_{\mb W}O(n)}\mb F)_{ij}=\dfrac{1}{2}(f_{ij}-w_{il}f_{kl}w_{kj}).
	\end{align*}
	By \eqref{eq:P1}, we derive
	\begin{align*}
		P_{ijkl}f_{kl}=\dfrac{1}{2}(\delta_{ik}\delta_{jl}-w_{kj}w_{il})f_{kl}=\dfrac{1}{2}(f_{ij}-w_{il}f_{kl}w_{kj}).
	\end{align*}
	Thus
	\begin{align*}
		(\mathcal{P}_{T_{\mb W}O(n)}\mb F)_{ij}=P_{ijkl}f_{kl}=\dfrac{1}{2}(f_{ij}-w_{il}f_{kl}w_{kj}).
	\end{align*}
	(ii) is obvious:
	\begin{align*}
		P_{ijkl}=\dfrac{1}{2}(\delta_{ik}\delta_{jl}-w_{kj}w_{il})=\dfrac{1}{2}(\delta_{ki}\delta_{lj}-w_{il}w_{kj})=P_{klij}.
	\end{align*}
	
	For (iii), we have
	\begin{align*}
		P_{ijkl}P_{klrs}&=\dfrac{1}{4}(\delta_{ik}\delta_{jl}-w_{kj}w_{il})(\delta_{kr}\delta_{ls}-w_{rl}w_{ks})\\
		&=\dfrac{1}{4}(\delta_{ir}\delta_{js}-w_{rj}w_{is}-w_{rj}w_{is}+w_{kj}w_{il}w_{rl}w_{ks}).
	\end{align*}
	Because $\mb W\in O(n)$, we have
	\begin{align*}
		w_{il}w_{rl}=\delta_{ir}, w_{kj}w_{ks}=\delta_{js}.
	\end{align*}
	Thus
	\begin{align*}
		P_{ijkl}P_{klrs}=\dfrac{1}{2}(\delta_{ir}\delta_{js}-w_{rj}w_{is})=P_{ijrs}.
	\end{align*}
\end{proof}
\begin{remark}
	We define $P_{ijkl}$ as in \eqref{def:tensorP} instead of \eqref{eq:P1}, because in this case, even though $\mb W\notin O(n)$, the following equality still holds:
	\begin{align} \label{eq:wP}
		w_{ij}P_{tjkl}+w_{tj}P_{ijkl}=0.
	\end{align}
	This is helpful in proving invariance of Stiefel manifold, i.e. Lemma \ref{prop:orthogonality}.
\end{remark}
Now we are ready to check that the Stiefel manifold is invariant for \eqref{eq:diffusionapprox}.

\begin{lemma} \label{prop:orthogonality}
	Consider \eqref{eq:diffusionapprox}. If $\mb{W_0}\in O(n)$, then $\mb W(t)\in O(n)$ holds for all $t>0$.
\end{lemma}
\begin{proof}
	By Lemma \ref{lmm:projection}, we can rewrite \eqref{eq:diffusionapprox} as
	\begin{align*}
		\dd w_{ij} = g_{ij}(\mb A, \mb W)\dd t+\eta P_{ijkl}(\mb W)f_{kl}(\mb W)\dd t+P_{ijkl}(\mb W)H_{klrs}(\mb W)\circ \dd B_{rs}.
	\end{align*}
	Then for $1\leq i,t\leq n$, remember that the SDE is in the Stratonovich sense, we have
	\begin{align*}
		\dd(w_{ij}w_{tj}) &= w_{ij}\circ\dd w_{tj} + w_{tj}\circ\dd w_{ij}\\
		&= (w_{ij}g_{tj}(\mb A,\mb W)+w_{tj}g_{ij}(\mb A,\mb W))\dd t +(w_{ij}P_{tjkl}+w_{tj}P_{ijkl})(f_{kl}\dd t+H_{klrs}\circ \dd B_{rs}).
	\end{align*}
	Notice that
	\begin{align*}
		w_{ij}P_{tjkl}+w_{tj}P_{ijkl} &= \dfrac{1}{2}[w_{ij}(w_{ts}w_{ks}\delta_{jl}-w_{kj}w_{tl})+w_{tj}(w_{is}w_{ks}\delta_{jl}-w_{kj}w_{il})]\\
		&= \dfrac{1}{2}(w_{il}w_{ts}w_{ks}-w_{ij}w_{kj}w_{tl}+w_{is}w_{ks}w_{tl}-w_{il}w_{tj}w_{kj})=0.
	\end{align*}
	Thus
	\begin{align*}
		\dd(w_{ij}w_{tj}) = (w_{ij}g_{tj}(\mb A,\mb W)+w_{tj}g_{ij}(\mb A,\mb W))\dd t,
	\end{align*}
	which is an ODE system. We can rewrite this in matrices as
	\begin{align*}
		\dfrac{\dd (\mb W\mb W^T)}{\dd t} &= \mb W\mb{G}^T(\mb A,\mb W) + \mb G(\mb A, \mb W)\mb W^T\\
		&= \mb W\mb W^T\mb A + \mb A\mb W\mb W^T - 2\mb W\mb W^T\mb A\mb W\mb W^T.
	\end{align*}
	Thus $\mb X=\mb W\mb W^T$ satisfies the algebraic Riccati equation:
	\begin{align*}
		\dfrac{\dd \mb X}{\dd t} = \mb A\mb X+\mb X\mb A-2\mb X\mb A\mb X,
	\end{align*}
	with initial value $\mb X=\mb{I_n}$. Thus by results in \cite{yan1994global}, we know that $\mb X(t)=\mb{I_n}$ holds. Therefore, $\mb W(t)\in O(n)$.
\end{proof}

\subsubsection{Fokker-Planck equation and Kolmogorov equation}
In the It\^o sense, \eqref{eq:diffusionapprox} reads as
\begin{align} \label{eq:diffIto}
\dot{\mb W}=\mb{G}(\mb{A},\mb{W})+\eta\mathcal{P}_{T_{\mb W}O(n)}\mb{F}(\mb{W})+\dfrac{\eta}{2}\mb{J}(\mb{W}) + \sqrt\eta\mathcal{P}_{T_{\mb W}O(n)}\dot{\mb Y}(\mb W),
\end{align}
Here $\mb Y$ is defined as
\begin{align}
	\dot{\mb Y}(\mb W)=(\dot{Y}_{ij})_{n\times n},\ \dot{Y}_{ij}(\mb W)=H_{ijkl}(\mb W)\dot{B}_{kl}.
\end{align}
The correction term $\mb J$ is then given by
\begin{equation}\label{eq:correction}
	\begin{aligned} 
	\mb K&=(K_{ijkl})_{n\times n \times n\times n},\ K_{ijkl}=P_{ijrs}H_{rskl}\\
	\mb J&=(J_{ij})_{n\times n},\ 
	J_{ij}(\mb{W}):=K_{rsml}\dfrac{\p K_{ijml}}{\p w_{rs}}.
	\end{aligned}
\end{equation}

Then \eqref{eq:diffIto} could be rewritten as
\begin{align*}
	\dot{w_{ij}}&=g_{ij}(\mb{A},\mb{W})+\eta P_{ijkl}f_{kl}+\dfrac{\eta}{2}J_{ij} + \sqrt\eta K_{ijkl}\dot{B_{kl}}\\
	&=g_{ij}(\mb{A},\mb{W})+\eta P_{ijkl}f_{kl}+\dfrac{\eta}{2}J_{ij} + \sqrt\eta P_{ijkl}H_{klrs}\dot{B_{rs}}
\end{align*}

The corresponding backward Kolmogorov equation of \eqref{eq:diffIto} is
\begin{align} \label{eq:backward}
	\p_tu(\mb{W},t) &= \mathcal{L}u(\mb{W},t),\ u(\mb{W},0)=\vp(\mb{W}),
\end{align}
where $\mathcal{L}$ is an elliptic operator defined as
\begin{equation}\label{eq:L}
	\begin{aligned} 
	\mathcal{L}u&:= \left(g_{ij}(\mb{A},\mb{W})+\eta P_{ijkl}f_{kl}+\dfrac{\eta}{2}J_{ij}\right)\dfrac{\p u}{\p w_{ij}} + \dfrac{\eta}{2}K_{ijrs}K_{klrs}\dfrac{\p^2 u}{\p w_{ij}w_{kl}}\\
	&=\left(g_{ij}(\mb{A},\mb{W})+\eta P_{ijkl}f_{kl}+\dfrac{\eta}{2}J_{ij}\right)\dfrac{\p u}{\p w_{ij}} +\dfrac{\eta}{2}P_{ijzv}H_{zvrs}P_{klxy}H_{xyrs}\dfrac{\p^2 u}{\p w_{ij}w_{kl}}.
	\end{aligned}
\end{equation}

The solution of \eqref{eq:backward} is given by
\begin{align} \label{eq:solution}
	u(\mb{W},t)=e^{t\mathcal{L}}\vp(\mb{W})=\E_{\mb{W}}[\vp(\mb{X}(t))],
\end{align}
here $\mb{X}(t)\in O(n)$ is the random process (before vectorization) determined by \eqref{eq:diffusionapprox} (or equivalently \eqref{eq:diffIto}) with starting point $\mb{X}(0)=\mb{W}\in O(n)$. For more details, we refer readers to \cite{oksendal2013stochastic}.

On the other hand, the solution of the forward Kolmogorov equation (the Fokker-Planck equation) is denoted as $\rho({\mb{W}},t)=e^{t\mathcal{L}^*}\rho_0(\mb{W})$:
\begin{align} \label{eq:forward}
	\p_t\rho = \mathcal{L}^*\rho,\ \rho({\mb{W}},t)=\rho_0({\mb{W}}),
\end{align}
where $\rho_0$ is the initial distribution of $\mb{W}$ and $\mathcal{L}^*$ is defined as
\begin{equation}\label{eq:L^*}
	\begin{aligned} 
	\mathcal{L}^*\rho&:=-\dfrac{\p}{\p w_{ij}}\left[\left(g_{ij}(\mb{A},\mb{W})+\eta P_{ijkl}f_{kl}+\dfrac{\eta}{2}J_{ij}\right) \rho\right]+ \dfrac{\eta}{2}\dfrac{\p^2 }{\p w_{ij}w_{kl}}\left(\rho K_{ijrs}K_{klrs}\right)\\
	&=-\dfrac{\p}{\p w_{ij}}\left[\left(g_{ij}(\mb{A},\mb{W})+\eta P_{ijkl}f_{kl}+\dfrac{\eta}{2}J_{ij}\right) \rho\right]+ \dfrac{\eta}{2}\dfrac{\p^2 }{\p w_{ij}w_{kl}}\left(\rho P_{ijzv}H_{zvrs}P_{klxy}H_{xyrs}\right).
	\end{aligned}
\end{equation}

The solution $\rho(\mb{W},t)=e^{t\mathcal{L}^*}\rho_0(\mb{W})$ is interpreted as the probability distribution of the random process $\mb{X}$ in \eqref{eq:diffusionapprox} with the initial distribution $\rho_0$.

Notice that $\{e^{t\mathcal{L}}\}_{t\geq 0}$ and $\{e^{t\mathcal{L}^*}\}_{t\geq 0}$ form two semigroups, we have the following basic properties:
\begin{lemma} \label{cor:contraction}
	Consider $\{e^{t\mathcal{L}}\}_{t\geq 0}$ in \eqref{eq:L} and $\{e^{t\mathcal{L}^*}\}_{t\geq 0}$ in \eqref{eq:L^*}. Then:
	\begin{enumerate}
		\item (probability and positivity preserving) if $\rho_0\in L^1(O(n))$ and $\rho_0\geq 0$, then $e^{t\mathcal{L}^*}\rho_0\geq 0$ and 
		\begin{align}
			\|e^{t\mathcal{L}^*}\rho\|_{L^1(O(n))}=\|\rho_0\|_{L^1(O(n))}.
		\end{align}
		In particular, $e^{t\mathcal{L}^*}$ is a contraction from $L^1(O(n))$ to $L^1(O(n))$.
		\item ($L^\infty$ contraction) Suppose that $\varphi:O(n)\to R$ is a continuous function. $e^{t\mathcal{L}}:L^{\infty}(O(n))\to L^{\infty}(O(n))$ is a contraction, i.e.
		\begin{align} \label{eq:contraction}
		\|e^{t\mathcal{L}}\vp\|_{L^{\infty}(O(n))}\leq \|\vp\|_{L^{\infty}(O(n))}.
		\end{align}   
	\end{enumerate}
\end{lemma}
\begin{proof}
	The proof of (1) can be found in \cite{feng2018semigroups}, here we just prove (2). Because $\vp$ is continuous and $O(n)$ is compact, we have $\varphi\in L^{\infty}(O(n))$.  Suppose $\mb{W_0}\in O(n)$, let $\mb{W}(t)$ be the stochastic process generated by \eqref{eq:diffusionapprox} starting from $\mb{W_0}$. Then we have 
	\begin{align}
	e^{t\mathcal{L}}\vp(\mb{W_0})=\E_{\mb{W_0}}\left[\vp(\mb{W}(t))\right].
	\end{align}
	By Lemma \ref{prop:orthogonality}, we know that $\mb{W}(t)\in O(n)$ for all $t\geq 0$, thus
	\begin{align}
	|e^{t\mathcal{L}}\vp(\mb{W_0})|\leq \E_{\mb{W_0}}|\vp(\mb{W}(t))|\leq\|\vp\|_{L^{\infty}(O(n))}. 
	\end{align}
	This holds for any $\mb{W_0}\in O(n)$, thus \eqref{eq:contraction} holds.
\end{proof}

More discussion on contraction and positivity preserving is available in \cite{soize1994fokker}.

Now we are ready to verify that \eqref{eq:diffusionapprox} serves as a weak diffusion approximation of the SGA iteration.

\subsubsection{First-order diffusion approximations} The method to validate that \eqref{eq:diffusionapprox} serves as a diffusion approximation originates from the idea of the Lax equivalence theorem \cite{lax1956survey}, which was first adopted in \cite{feng2018semigroups} for the same purpose.
\begin{theorem} \label{thm:diffusionapprox}
	(first-order diffusion approximation) Fix time $T>0$. Consider $\{u^k(\mb{W})\}_{k\geq 0}$ defined in \eqref{def:u}. Suppose that $\vp\in C_b^4(\R^{n\times n})$ and $u(\mb{W},t)$ is the solution to \eqref{eq:backward} with the initial value $u(\cdot,0)=\vp(\cdot)$. Then there exist $\eta_1(M,T)>0$ and $C_1(M,T,\vp)>0$ such that for any $\eta\in(0,\ \eta_1(M,T)]$,
	\begin{align}
	\sup\limits_{\mb{W}\in O(n),k\eta\leq T}|u^k(\mb{W})-u(\mb{W},k\eta)|\leq C_1(M,T,\vp)\eta.
	\end{align} 
Here $M$ is the constant in \eqref{ass:L_inf}.
\end{theorem}
\begin{proof}
	In the following proof, $C$ is a general constant that varies among equations.
	
	By Theorem \ref{thm:semigroup}, there exist constants $\eta(M,T)$ and $C(M,T)$ such that for all $k=0,1,2...,[T/\eta]$ and $\eta\in(0,\ \eta(M,T)]$,
	\begin{align} \label{eq:regularity}
		\|u^k\|_{C^4(B(0,\sqrt{n+1}))}\leq C(M,T)\|\vp\|_{C^4(\R^{n\times n})}.
	\end{align}
	Then by Taylor's expansion w.r.t. $\eta$, we have that for any $\mb{W}\in O(n)$,
	\begin{equation}\label{eq:Taylor1}
		\begin{aligned} 
		\left|u^{k+1}(\mb{W})-u^k(\mb{W})-\eta g_{ij}(\mb A,\mb W)\dfrac{\p u^k}{\p w_{ij}}\right| &\leq C(M,T)\|u^k\|_{C^4(B(0,\sqrt{n+1}))}\eta^2\\
		&\leq C(M,T,\vp)\eta^2.
		\end{aligned}
	\end{equation}
	Details of this computation can be found in Section \ref{sec:appendix}.
	
	By Lemma \ref{cor:contraction}, $e^{t\mathcal{L}}$ is a contraction from $L^{\infty}(O(n))$ to itself. Therefore for any $\mb{W}\in O(n)$, there exists $\eta'\in(0,\eta]$ such that 
	\begin{equation*}
		\begin{aligned}
		|e^{\eta\mathcal{L}}u^k(\mb{W})-u^k(\mb{W})-\eta\mathcal{L}u^k(\mb{W})|&\leq \dfrac{\eta^2}{2}|e^{\eta'\mathcal{L}}\mathcal{L}^2u^k(\mb{W})|\\
		&\leq \dfrac{\eta^2}{2}\|\mathcal{L}^2u^k\|_{L^{\infty}(O(n))}\\
		&\leq C(M,T,\|u^k\|_{C^4(B(0,\sqrt{n+1}))})\eta^2\\
		&\leq C(M,T,\vp)\eta^2.
		\end{aligned}
	\end{equation*}
	Recall the definition of $\mathcal{L}$ in \eqref{eq:L}, we can rewrite the above estimate as
	\begin{align} \label{eq:Taylor2}
		\left|e^{\eta\mathcal{L}}u^k(\mb{W})-u^k(\mb{W})-\eta g_{ij}(\mb A,\mb W)\dfrac{\p u^k}{\p w_{ij}}\right|\leq C(M,T,\vp)\eta^2
	\end{align}
	See Section \ref{sec:appendix} for details of this computation.
	
	Thus by \eqref{eq:Taylor1} and \eqref{eq:Taylor2}, we have
	\begin{align}
		\|u^{k+1}(\mb{W})-e^{\eta\mathcal{L}}u^k(\mb{W})\|_{L^{\infty}(O(n))}\leq C(M,T,\vp)\eta^2.
	\end{align}
	 
	Now for $k=0,1,2...,[T/\eta]$, define
	\begin{align}
		e_k:=\|u^k(\mb{W})-u(\mb{W},k\eta)\|_{L^{\infty}(O(n))},\ r_k:=\|u^{k+1}(\mb{W})-e^{\eta\mathcal{L}}u^k(\mb{W})\|_{L^{\infty}(O(n))}.
	\end{align}
	Thus $r_k\leq C(M,T,\vp)\eta^2$ for $k=1,2,...,[T/\eta]$. Notice that $u(\mb{W},(k+1)\eta)=e^{\eta\mathcal{L}}u(\mb{W},k\eta)$, we have
	\begin{equation} \label{eq:est}
		\begin{aligned}
		e_{k+1} &= \|u^{k+1}(\mb{W})-u(\mb{W},(k+1)\eta)\|_{L^{\infty}(O(n))}\\
		&= \|u^{k+1}(\mb{W})-e^{\eta\mathcal{L}}u(\mb{W},k\eta)\|_{L^{\infty}(O(n))}\\
		&\leq \|u^{k+1}(\mb{W})-e^{\eta\mathcal{L}}u^k(\mb{W})\|_{L^{\infty}(O(n))}+\|e^{\eta\mathcal{L}}(u(\mb{W},k\eta)-u^k(\mb{W}))\|_{L^{\infty}(O(n))}\\
		&\leq r_k+e_k.
		\end{aligned}
	\end{equation}
	The last step in \eqref{eq:est} used that $e^{\eta\mathcal{L}}$ is a contraction on $L^{\infty}(O(n))$. Thus summing up \eqref{eq:est} in $k$ yields
	\begin{align}
		e_k\leq \sum_{l=1}^{k-1}r_l \leq\dfrac{TC''(M,T,\vp)}{\eta}\cdot\eta^2=C_1(M,T,\vp)\eta.
	\end{align} 
	This holds uniformly in $k$, so the claim is proven.
\end{proof}

\subsubsection{Unstable second order approximation}

According to the proof of Theorem \ref{thm:diffusionapprox}, the key step that ensures first order approximation is the following estimate on the truncation error, which is of second order:
\begin{align*}
	\|S\varphi(\mb W)-e^{\eta\mathcal{L}}\varphi(\mb W)\|_{L^{\infty}(O(n))}\leq C\eta^2.
\end{align*}
Here $\mathcal{L}$ is defined in \eqref{eq:L}.

To derive second order approximation, we need to carefully select the drift term $\mb F$ and $\mb H$ in \eqref{eq:diffusionapprox} such that
\begin{align}
	\|S\varphi(\mb W)-e^{\eta\mathcal{L}}\varphi(\mb W)\|_{L^{\infty}(O(n))}\leq C\eta^3.
\end{align}

Direct calculation yields
\begin{equation*}
\begin{aligned}
S\vp(\mb{W}) &= \E\left[\vp(\mb{W}+\eta\mb{G}(\mb{x}\mb{x}^T,\mb{W}))\right]\\
&= \vp(\mb{W}) +\eta g_{ij}(\mb A,\mb{W})\dfrac{\p\vp}{\p w_{ij}}+\dfrac{\eta^2}{2}\E\left[g_{ij}(\mb x\mb x^T,\mb W)g_{kl}(\mb x\mb x^T,\mb W)\right]\dfrac{\p^2 \varphi}{\p w_{ij}\p w_{kl}}+O(\eta^3).
\end{aligned}
\end{equation*}

Meanwhile,
\begin{equation*}
	\begin{aligned}
	e^{\eta\mathcal{L}}\varphi(\mb W) = \varphi(\mb W)+\eta\mathcal{L}\varphi(\mb W) +\dfrac{1}{2}\eta^2\mathcal{L}^2\varphi(\mb W)+O(\eta^3). 
	\end{aligned}
\end{equation*}

By \eqref{eq:L}, we have
\begin{align}
	\begin{aligned}
	\mathcal{L}^2\varphi &=\mathcal{L}\left[g_{ij}(\mb A,\mb W)\dfrac{\p\vp}{\p w_{ij}}+O(\eta)\right]\\
	&=g_{kl}(\mb A,\mb W)\dfrac{\p }{\p w_{kl}}\left(g_{ij}(\mb A,\mb W)\dfrac{\p\vp}{\p w_{ij}}\right)+O(\eta)\\
	&=\left(g_{kl}(\mb A,\mb W)\dfrac{\p g_{ij}(\mb A,\mb W)}{\p w_{kl}}\right)\dfrac{\p\vp}{\p w_{ij}}+g_{ij}(\mb A,\mb W)g_{kl}(\mb A,\mb W)\dfrac{\p^2\vp}{\p w_{ij}\p w_{kl}}+O(\eta).
	\end{aligned}
\end{align}

Comparing the coefficients of $\dfrac{\p^2\vp}{\p w_{ij}\p w_{kl}}$ and $\dfrac{\p \vp}{\p w_{ij}}$, we derive

\begin{equation}\label{eq:second} 
\begin{aligned}
P_{ijkl}f_{kl}&=-\dfrac{1}{2}J_{ij}-\dfrac{1}{2}g_{kl}(\mb A,\mb W)\dfrac{\p g_{ij}(\mb A,\mb W)}{\p w_{kl}},\\
P_{ijuv}H_{uvrs}P_{klxy}H_{xyrs} &= \E[g_{ij}(\mb A-\mb x\mb x^T,\mb W)g_{kl}(\mb A-\mb x\mb x^T,\mb W)].
\end{aligned}
\end{equation}

Here $\mb F=(f_{ij})_{n\times n}$ and $\mb H=(H_{ijkl})_{n\times n\times n\times n}$ are coefficients in \eqref{eq:diffusionapprox}. Details of the derivation of \eqref{eq:second} is summarized in Section \ref{sec:appendix}.

To interpret \eqref{eq:second}, one can see that the R.H.S. of the second equation in \eqref{eq:second} is exactly the covariance tensor of $\mb G(\mb x\mb x^T,\mb W)$. We denote it as
\begin{align} \label{eq:cov}
\mb M(\mb W) = (M_{ijkl}(\mb W))_{n\times n\times n\times n},\ M_{ijkl} = \E\left[g_{ij}(\mb x\mb x^T-\mb A,\mb W)g_{kl}(\mb x\mb x^T-\mb A,\mb W)\right].
\end{align}

Therefore, we desire suitable $\mb H$ such that
\begin{align}\label{eq:good}
	P_{ijuv}H_{uvrs}P_{klxy}H_{xyrs} = M_{ijkl}.
\end{align}
This can be realized if $\mb W\in O(n)$, which is sufficient for our purpose. 

\begin{lemma} \label{lmm:orth}
Consider $\mb M$ in \eqref{eq:cov}. Then there exists a unique $\mb N=(N_{ijkl})_{n\times n\times n\times n}$ that satisfy
\begin{enumerate}[(i)]
	\item (symmetry) $N_{ijkl}=N_{klij}$;
	\item (positive semidefinite) for any $(m_{ij})_{n\times n}$, $m_{ij}N_{ijkl}m_{kl}\geq 0;$
	\item (square root of $\mb M$) $N_{ijrs}N_{klrs}=M_{ijkl}$.
\end{enumerate} 
Moreover, if $\mb W\in O(n)$, then 
\begin{align}
	P_{ijrs}N_{rskl} = N_{ijkl},
\end{align}
i.e. $\mb {PN}=\mb N$.
\end{lemma}
See Section \ref{sec:appendix} for the proof of this lemma. In the view of Lemma \ref{lmm:orth}, we also denote $\mb N$ as
\begin{align}\label{eq:squareroot}
	\mb N=\sqrt{\mb M}.
\end{align}

If we take
\begin{align*}
	\mb H=\mb N=\sqrt{\mb M},
\end{align*} 
in \eqref{eq:diffusionapprox}, then 
\begin{align*}
	P_{ijuv}H_{uvrs}P_{klxy}H_{xyrs} = P_{ijuv}N_{uvrs}P_{klxy}N_{xyrs} = N_{ijrs}N_{klrs} = M_{ijkl}
\end{align*}
holds if $\mb W\in O(n)$. Thus \eqref{eq:good} is satisfied. 

Now we take
\begin{align}\label{eq:condition2}
	\mb F=\mb 0, \mb H=\mb N=\sqrt{\mb M}
\end{align} 
in \eqref{eq:diffusionapprox}. By Lemma \ref{lmm:orth}, we know that \eqref{eq:diffusionapprox} stays on $O(n)$, so \eqref{eq:diffusionapprox} could be rewritten as
\begin{equation}\label{eq:secondSDE}
	\begin{aligned}
	\dd w_{ij} &= g_{ij}(\mb A, \mb W)\dd t+\sqrt{\eta}P_{ijkl}(\mb W)N_{klrs}(\mb W)\circ \dd B_{rs},\\
	&= g_{ij}(\mb A, \mb W)\dd t+\sqrt{\eta}N_{ijrs}(\mb W)\circ \dd B_{rs}.
	\end{aligned}
\end{equation}

This SDE seems to serve as the second order approximation for the SGA method: we have $g_{ij}$ in \eqref{eq:SGA} as the drift term; covariance of $g_{ij}$ is also reflected in the Brownian motion. However, because $\mb F$ in \eqref{eq:condition2} does not satisfy \eqref{eq:second}, \eqref{eq:secondSDE} is not a second order approximation, but a first order approximation by Theorem \ref{thm:diffusionapprox}.

Moreover, there is even no solution to the first equation of \eqref{eq:second}. This excludes the possibility of deriving a second order approximation on the Stiefel manifold. 

Even with the last try, we require $\mb H$ satisfy \eqref{eq:second} and just replace $\mathcal{P}_{T_{\mb W}O(n)}\mb F$ in \eqref{eq:diffusionapprox} by
\begin{align} \label{eq:flowL}
	\mb L =(L_{ij})_{n\times n},\ L_{ij}:= -\dfrac{1}{2}J_{ij}-\dfrac{1}{2}g_{kl}(\mb A,\mb W)\dfrac{\p g_{ij}(\mb A,\mb W)}{\p w_{kl}},
\end{align} 
then resulted SDE still does not stay on the Stiefel manifold. Therefore, we have no second order diffusion approximation that stays on the Stiefel manifold.

\begin{lemma}\label{lmm:unstable}
	(unstable second order approximation) Consider \eqref{eq:second}, $\mb L$ in \eqref{eq:flowL}, and $\mb J=(J_{ij})_{n\times n}$ in \eqref{eq:correction}. Then 
	\begin{enumerate}[(i)]
		\item Suppose that $\mb H$ satisfies \eqref{eq:second}. Then there is no $\mb F=(f_{ij})_{n\times n}$ that satisfies
		\begin{align*}
			P_{ijkl}f_{kl}=L_{ij},
		\end{align*} 
		i.e. the first equation of \eqref{eq:second} admits no solution.
		\item Consider the solution $\mb W(t)$ to the following SDE:
		\begin{align}\label{eq:bad}
				\dd w_{ij} = \left(g_{ij}(\mb A, \mb W)+\eta L_{ij}\right)\dd t+\sqrt{\eta}N_{ijrs}(\mb W)\circ \dd B_{rs},\ \mb W(0)=\mb{W_0}\in O(n).
		\end{align}
		Then there is no time interval $[t_0,t_1]$ such that $\mb W(t)\in O(n)$ for $t\in [t_0,t_1]$. Here $\mb N=(N_{ijkl})_{n\times n\times n\times n}$ is the square root of the covariance matrix $\mb M$ (see Lemma \ref{lmm:orth}).
	\end{enumerate}
\end{lemma}
\begin{proof}
	We first prove (i) by contradiction. Suppose there is a solution $\mb F$, then by \eqref{eq:wP}, 
	\begin{align*}
		w_{ij}L_{tj}+w_{tj}L_{ij} = w_{ij}P_{tjkl}f_{kl}+w_{tj}P_{ijkl}f_{kl} = 0.
	\end{align*}
	Notice 
	\begin{align*}
		-2(w_{ij}L_{tj}+w_{tj}L_{ij}) = \underbrace{w_{ij}J_{tj}+w_{tj}J_{lj}}_{\text I} + \underbrace{g_{kl}(\mb A, \mb W)\left(w_{ij}\dfrac{\p g_{tj}(\mb A,\mb W)}{\p w_{kl}}+w_{tj}\dfrac{\p g_{ij}(\mb A,\mb W)}{\p w_{kl}}\right)}_{\text{II}}.
	\end{align*}
	By definition of $\mb J,\mb K$ in \eqref{eq:correction}, we have
	\begin{align*}
		\text I = \dfrac{\p}{\p w_{rs}}\left(\left(w_{ij}P_{tjuv}+w_{tj}P_{ijuv}\right)H_{uvml}K_{rsml}\right)-K_{tjml}\dfrac{\p}{\p w_{rs}}(w_{ij}K_{rsml})-K_{ijml}\dfrac{\p}{\p w_{rs}}(w_{tj}K_{rsml}).
	\end{align*}
	By \eqref{eq:wP}, the first term of I is zero, so we have
	\begin{align*}
		\text I &=-K_{tjml}\dfrac{\p}{\p w_{rs}}(w_{ij}K_{rsml})-K_{ijml}\dfrac{\p}{\p w_{rs}}(w_{tj}K_{rsml})\\
		&=-2K_{tjml}K_{ijml} - w_{ij}P_{tjuv}H_{uvml}K_{rsml} - w_{tj}P_{ijuv}H_{uvml}K_{rsml}\\
		&=-2K_{tjml}K_{ijml}\\
		&=-2\E[g_{tj}(\mb A-\mb x\mb x^T,\mb W)g_{ij}(\mb A-\mb x\mb x^T,\mb W)].
	\end{align*}
	For II, we have 
	\begin{align*}
		\text{II} &= \dfrac{\p}{\p w_{kl}}[(w_{ij}g_{tj}(\mb A,\mb W)+w_{tj}g_{ij}(\mb A,\mb W))g_{kl}(\mb A,\mb W)] \\
		&\ \ \ \ - g_{tj}(\mb A,\mb W)\dfrac{\p}{\p w_{kl}}(w_{ij}g_{kl}(\mb A,\mb W)) - g_{ij}(\mb A,\mb W)\dfrac{\p }{\p w_{kl}}(w_{tj}g_{kl}(\mb A,\mb W))\\
		&= -2g_{ij}(\mb A,\mb W)g_{tj}(\mb A,\mb W) + g_{kl}\dfrac{\p}{\p w_{kl}}(g_{tj}w_{ij}+g_{ij}w_{tj}).
	\end{align*}
	Thus
	\begin{align*}
		-2(w_{ij}L_{tj}+w_{tj}L_{ij}) &= \text I+ \text{II}\\
		&= -2\E[g_{tj}(\mb A-\mb x\mb x^T,\mb W)g_{ij}(\mb A-\mb x\mb x^T,\mb W)] -2g_{ij}(\mb A,\mb W)g_{tj}(\mb A,\mb W) \\
		&\ \ \ \ + g_{kl}\dfrac{\p}{\p w_{kl}}(g_{tj}w_{ij}+g_{ij}w_{tj})\\
		&= -2\E[g_{tj}(\mb x\mb x^T,\mb W)g_{ij}(\mb x\mb x^T,\mb W)]+ g_{kl}\dfrac{\p}{\p w_{kl}}(g_{tj}w_{ij}+g_{ij}w_{tj}).
	\end{align*}
	This can not be zero for general R.V. $\mb x$, because the first term depends on the fourth-order momentum while the second term only depends on the second-order momentum. This is a contradiction, so 
	$w_{ij}L_{tj}+w_{tj}L_{ij}\neq 0$ and \eqref{eq:second} admits no solution.
	
	For (ii), we still prove by contradiction. Suppose that for some time interval $[t_0,t_1]$, $\mb W(t)\in O(n)$. Then 
	\begin{align*}
		\dfrac{\dd (\mb W\mb W^T)}{\dd t} =0,\ t\in (t_0,t_1).
	\end{align*}
	By Lemma \ref{lmm:orth}, we know that if $\mb W\in O(n)$, then $P_{ijrs}N_{rskl}=N_{ijkl}$. So we can rewrite \eqref{eq:bad} as
	\begin{align*}
		\dd w_{ij} = \left(g_{ij}(\mb A, \mb W)+\eta L_{ij}\right)\dd t+\sqrt{\eta}P_{ijrs}N_{rskl}\circ \dd B_{kl}
	\end{align*}
	However, 
	\begin{align*}
		\dd (w_{ij}w_{tj}) &= w_{ij}\dd w_{tj} + w_{tj}\dd w_{ij}\\
		&= (w_{ij}g_{tj}(\mb A,\mb W)+w_{tj}g_{ij}(\mb A,\mb W))\dd t + \eta(w_{ij}L_{tj}+w_{tj}L_{ij})\dd t\\&\ \ \ \ + \sqrt{\eta} (w_{ij}P_{tjrs}+w_{tj}P_{ijrs})N_{rskl}\circ B_{kl}.
	\end{align*}
	Remember that if $\mb W\in O(n)$, then $\mb G(\mb A,\mb W)=\mathcal{P}_{T_\mb WO(n)}\mb G(\mb A,\mb W)$. So $g_{ij}=P_{ijkl}g_{kl}$, which yields that for all $t\in (t_0,t_1)$, 
	\begin{align*}
		w_{ij}g_{tj}+ w_{tj}g_{ij} = (w_{ij}P_{tjrs}+w_{tj}P_{ijrs})g_{rs} = 0
	\end{align*}
	by \eqref{eq:wP}. Due to the same reason, 
	\begin{align*}
		\sqrt{\eta} (w_{ij}P_{tjrs}+w_{tj}P_{ijrs})N_{rskl}\circ B_{kl}=0.
	\end{align*}
	Thus
	\begin{align*}
		\dd(w_{ij}w_{tj}) = \eta(w_{ij}L_{tj}+w_{tj}L_{ij})\dd t.
	\end{align*}
	However, $(w_{ij}L_{tj}+w_{tj}L_{ij})\neq 0$ by (i). Thus $\dd (w_{ij}w_{tj})\neq 0$ for $t\in (t_0,t_1)$, which contradicts with
	\begin{align*}
		w_{ij}w_{tj}=\delta_{it}.
	\end{align*}
	This is a contradiction, the proof is finished.
\end{proof}

In the view of Lemma \ref{lmm:unstable}, \eqref{eq:bad} fails to stay on the Stiefel manifold. In this case, \eqref{eq:backward} is a degenerate parabolic PDE with unbounded coefficients, which fails to control the diffusion in the normal direction of the Stiefel manifold. Thus, utilizing solutions of \eqref{eq:backward} to approximate the behavior of the semigroup \eqref{def:semigroup} is meaningless.

\section{Reversible diffusion approximation:  exponential convergence} \label{sec:sde}

In Theorem \ref{thm:diffusionapprox}, we derived diffusion approximations on the Stiefel manifold. A natural question is that whether the SDE is ergodic and converges. 

In fact, reversibility and Poincare's inequality ensure exponential convergence. To see this, the Fokker-Planck operator \eqref{eq:L^*} can be recast as 
\begin{align*}
	\mathcal{L}^*\rho &= \dfrac{\p}{\p w_{ij}}\left[\rho\left(-g_{ij}-\eta P_{ijkl}f_{kl}+\dfrac{\eta}{2}K_{ijrs}K_{klrs}\dfrac{\p\log\rho}{\p w_{kl}}+\dfrac{\eta}{2}J_{ij}\right)\right]\\
	&= \dfrac{\p}{\p w_{ij}}\left[\dfrac{\eta}{2}K_{ijrs}K_{klrs}\rho\left(\dfrac{\eta J_{ij}-2g_{ij}-2\eta P_{ijkl}f_{kl}}{\eta K_{ijrs}K_{klrs}}+\dfrac{\p\log\rho}{\p w_{kl}}\right)\right].
\end{align*}
If there exists a function $U(\mb W)$ such that 
\begin{align} \label{eq:reversibility}
	\dfrac{\p U}{\p w_{kl}} =\dfrac{\eta J_{ij}-2g_{ij}-2\eta P_{ijkl}f_{kl}}{\eta K_{ijrs}K_{klrs}},
\end{align}
then solutions to $\p_t\rho=\mathcal{L}^*\rho$ satisfies
\begin{align*}
\p_t\rho = \dfrac{\p}{\p w_{ij}}\left[\dfrac{\eta}{2}K_{ijrs}K_{klrs}\rho\dfrac{\p(\log\rho+U)}{\p w_{kl}}\right]=\dfrac{\p}{\p w_{ij}}\left[\dfrac{\eta}{2}e^{-U}K_{ijrs}K_{klrs}\dfrac{\p(e^U\rho)}{\p w_{kl}}\right].
\end{align*}
Without loss of generality, we assume that $\int_{O(n)}e^{-U(\mb W)}\dd\mb W=1$. Then 
multiply $e^U\rho-1$ on both sides, and by definition of $\mb K$ in \eqref{eq:correction}, we derive
\begin{align*}
\dfrac{\dd}{\dd t}\int_{O(n)}e^{-U}|e^U\rho-1|^2\dd\mb W &= -\eta\int_{O(n)}e^{-U}\left(K_{ijrs}\dfrac{\p}{\p w_{ij}}(e^U\rho-1)\right)\left(K_{klrs}\dfrac{\p}{\p w_{kl}}(e^U\rho-1)\right)\dd\mb W\\
&= -\eta \int_{O(n)}e^{-U}|\mb H^T\nabla_{\mb W}(e^{U}\rho-1)|^2\dd\mb W.
\end{align*} 
Here $\nabla_{\mb W}$ is the gradient operator on $(O(n),g_e)$ where $g_e$ is the Euclidean metric. Then by Poincare's inequality, we derive exponential convergence. 

Therefore, we desire to carefully select $\mb H$ and $\mb F$ such that the potential condition \eqref{eq:reversibility} is satisfied. We provide the following two special cases where reversibility is ensured.
 
First, we consider the overdamped Langevin on $O(n)$. In particular, we select the potential as
$U(\mb W;\mb A,\mb N)=\mathrm{tr}(\mb N\mb W^T\mb A\mb W)$, where $\mb N$ is a diagonal matrix. Then we recover the Oja-Brockett flow with Bronwian motion. In fact, the Oja-Brockett flow is the gradient flow of $U(\mb W;\mb A,\mb N)=\mathrm{tr}(\mb N\mb W^T\mb A\mb W)$ on $(O(n),g_e)$.

Second, we consider the two-dimensional case, i.e. $n=2$. In this case, orthogonal matrices are determined by the rotational angle. The SDE of the angle is an SDE on $\R$, which automatically satisfy the potential condition.

In each case, Poincare's inequality is verified, so they converge to the invariant measure exponentially fast.  
\subsection{The overdamped Langevin dynamics on $O(n)$}\label{sec:OLE}
Suppose that $U(\mb{Q}):O(n)\to\R$ is a smooth function (which serves as the free energy), consider the overdamped Langevin dynamics on $O(n)$:
\begin{align} \label{eq:OLEStiefel}
\dd\mb{Q}(t) = \mathcal{P}_{T_{\mb{Q}}O(n)}\circ(-\nabla U(\mb{Q})\dd t+\sigma\dd\mb{W}(t)).
\end{align} 
Here $\mathcal{P}_{T_{\mb{Q}}O(n)}$ is the projection onto $T_{\mb{Q}}O(n)$, $\nabla$ represents the derivative w.r.t. $\mb{Q}$, i.e.
\begin{align*}
(\nabla U(\mb{Q}))_{ij} = \dfrac{\p u(\mb{Q})}{\p q_{ij}}.
\end{align*}
$\mb{W}(t)\in\R^{n\times n}$ is the standard Brownian motion and $\sigma$ is a constant. '$\circ$' means that the above SDE is in the Stratonovich sense. In general, $\sigma$ can be a matrix that depends on $\mb{Q}$. To illustrate the the Langevin dynamics on $O(n)$, we consider the simplest case here which is sufficient for diffusion approximation.

If we take $U(\mb Q)=-\mathrm{tr}(\mb N\mb{Q}^T\mb{AQ})$ where $\mb N$ is a diagonal matrix with entries on the diagonal line aligned in the descending order, then \eqref{eq:OLEStiefel} reads as
\begin{align} \label{eq:dtbOjaBrockett}
\dd\mb{Q}(t) = (\mb{AQN} -\mb{QNQ}^T\mb{AQ})\dd t +\sigma \mathcal{P}_{T_{\mb{Q}}O(n)}\circ\dd\mb{W}(t).
\end{align} 
Then \eqref{eq:dtbOjaBrockett} should be viewed as the disturbed Oja-Brockett flow \cite{brockett1991dynamical}.
 
By the conversion rule, \eqref{eq:OLEStiefel} should be formulated in Ito's sense as

\begin{align}
\dd\mb{Q}(t) = \left[-\mathcal{P}_{T_\mb{Q}O(n)}\nabla U(\mb{Q})-\dfrac{\sigma^2(n-1)}{4}\mb{Q}\right]\dd t + \sigma\mathcal{P}_{T_\mb{Q}O(n)}\dd\mb{W}(t).
\end{align}
By Ito's formula, we can derive the Fokker-Planck equation of \eqref{eq:OLEStiefel} which is 
\begin{align} \label{eq:FPOLE}
\dfrac{\p\rho(\mb{Q},t)}{\p t} = \nabla_{\mb{Q}}\cdot(\rho(\mb{Q},t)\nabla_{\mb{Q}}U(\mb{Q}))+\dfrac{\sigma^2}{2}\Delta_{\mb{Q}}\rho(\mb{Q},t).
\end{align}
Here $\nabla_{\mb{Q}}\cdot,\nabla_{\mb{Q}}$ and $\Delta_{\mb{Q}}$ are the divergence, gradient and Laplace-Beltrami operator on $(O(n),\ g_e)$ respectively. See Section \ref{sec:appendix} for more details.

Compactness of $O(n)$ implies exponential convergence of \eqref{eq:FPOLE}. Direct calculation yields that the invariant measure of \eqref{eq:FPOLE} is given by
\begin{align} \label{eq:OLEinvmeas}
\rho_{eq}(\mb{Q}):=\dfrac{1}{Z}e^{-2U(\mb{Q})/\sigma^2},\ Z:=\int_{O(n)}e^{-2U(\mb{Q})/\sigma^2}\dd{\mb{V}},
\end{align} 
where $\dd\mb{V}$ is the volume form on $(SO(n),g_e)$. Equation \eqref{eq:FPOLE} can be reformulated as 
\begin{align}
\dfrac{\p\rho(\mb{Q},t)}{\p t} = \dfrac{\sigma^2}{2}\nabla_{\mb{Q}}\cdot\left(\rho_{eq}(\mb{Q})\nabla_{\mb{Q}}\left(\dfrac{\rho(\mb{Q},t)}{\rho_{eq}(\mb{Q})}\right)\right):=\mathcal{L}^*\rho(\mb{Q},t),
\end{align}
and we denote the Fokker-Planck operator as $\mathcal{L}^*$. This also implies that the invariant measure of \eqref{eq:FPOLE} is unique since
\begin{equation*}
\begin{aligned}
\mathcal{L}^*\rho=0 &\iff \nabla_{\mb{Q}}\cdot\left(\rho_{eq}(\mb{Q})\nabla_{\mb{Q}}\left(\dfrac{\rho(\mb{Q},t)}{\rho_{eq}(\mb{Q})}\right)\right) = 0\\
&\iff \int_{O(n)}\rho_{eq}(\mb{Q})\left|\nabla_{\mb{Q}}\left(\dfrac{\rho(\mb{Q},t)}{\rho_{eq}(\mb{Q})}\right)\right|^2\dd{V}=0\\
&\iff \rho(\mb{Q})=c\rho_{eq}(\mb{Q}),
\end{aligned}
\end{equation*}
where $c$ is a constant. If $\rho(\mb{Q})$ is a probability measure on $SO(n)$, then $c=1$ and $\rho=\rho_{eq}$. In the above induction we used the positivity of $\rho_{eq}(\mb{Q})$, which is a consequence of the continuity of $U(\mb{Q})$ and the compactness of $SO(n)$.

Multiplying $\rho(\mb{Q},t)/\rho_{eq}(\mb{Q})-1$ on both sides of \eqref{eq:FPOLE} results in
\begin{align} \label{eq:estimate}
\dfrac{\dd }{\dd t}\int_{O(n)}\rho_{eq}(\mb{Q})\left|\dfrac{\rho(\mb{Q},t)-\rho_{eq}(\mb{Q})}{\rho_{eq}(\mb{Q})}\right|^2\dd\mb{V} = -{\sigma^2}\int_{O(n)}\rho_{eq}(\mb{Q})\left|\nabla_{\mb{Q}}\left(\dfrac{\rho(\mb{Q},t)-\rho_{eq}(\mb{Q})}{\rho_{eq}(\mb{Q})}\right)\right|^2\dd\mb{V}.
\end{align}
If we can prove the Poincare inequality in $L^2(\rho_{eq}(\mb{Q})\dd\mb{V})$, i.e., there exists a constant $C>0$ such that for all $f\in H^1(\rho_{eq}(\mb{Q})\dd\mb{V})$ satisfying $\int_{SO(n)}f(\mb{Q})\rho_{eq}(\mb{Q})\dd\mb{V}=0$, we have
\begin{align} \label{eq:Poincare2}
\int_{O(n)}\rho_{eq}(\mb{Q})\left|{\nabla_{\mb{Q}}\left(f(\mb{Q})\right)}\right|^2\dd\mb{V}\geq C\int_{O(n)}\rho_{eq}(\mb{Q})|f(\mb{Q})|^2\dd\mb{V},
\end{align}
then \eqref{eq:estimate} yields
\begin{align*}
\dfrac{\dd }{\dd t}\int_{O(n)}\rho_{eq}(\mb{Q})\left|\dfrac{\rho(\mb{Q},t)-\rho_{eq}(\mb{Q})}{\rho_{eq}(\mb{Q})}\right|^2\dd\mb{V} \leq -C\int_{O(n)}\rho_{eq}(\mb{Q})\left|\dfrac{\rho(\mb{Q},t)-\rho_{eq}(\mb{Q})}{\rho_{eq}(\mb{Q})}\right|^2\dd\mb{V}.
\end{align*}
By Gronwall's inequality, this gives exponential convergence of \eqref{eq:FPOLE} whose initial value is a probability measure. 

Now we prove the Poincare inequality in $L^2(\rho_{eq}\dd\mb{V})$.
\begin{lemma}(Poincare's inequality) Suppose that $f\in H^1(\rho_{eq}(\mb{Q})\dd\mb{V})$ satisfying $\int_{O(n)}f(\mb{Q})\rho_{eq}(\mb{Q})\dd\mb{V}=0$. Then \eqref{eq:Poincare2} holds.
\end{lemma} 
\begin{proof}
	To prove this, we prove that for any $\lambda>0$, $(\lambda-\mathcal{L}^*)^{-1}:\ L^2\left(\dd\mb{V}/\rho_{eq}(\mb{Q})\right)\to L^2\left(\dd\mb{V}/\rho_{eq}(\mb{Q})\right)$ is a compact operator. Suppose that $\{u_n\}_{n\geq 1}$ and $\{g_n\}_{n\geq 1}$ are two sequences in $L^2(\dd\mb{V}/\rho_{eq}(\mb{Q}))$ such that 
	\begin{align*}
	(\lambda-\mathcal{L}^*)u_n=g_n,\ n\geq 1,
	\end{align*}
	and $\{g_n\}_{n\geq 1}$ is uniformly bounded in $L^2(\dd\mb{V}/\rho_{eq}(\mb{Q}))$. Then multiplying $u_n$ on both sides yields
	\begin{align*}
	\int_{O(n)}\dfrac{1}{\rho_{eq}}\cdot u_n\cdot(\lambda-\mathcal{L}^*)u_n\dd\mb{V} &= \int_{O(n)}\rho_{eq}\left|\nabla_{\mb{Q}}\left(\dfrac{u_n}{\rho_{eq}}\right)\right|^2\dd\mb{V} + \lambda\int_{SO(n)}\rho_{eq}\left|\dfrac{u_n}{\rho_{eq}}\right|^2\dd\mb{V}\\
	&= \int_{O(n)}\dfrac{1}{\rho_{eq}}\cdot g_nu_n\dd\mb{V}.
	\end{align*}
	By Cauchy-Schwartz's inequality, we know that $u_n/\rho_{eq}$ uniformly bounded in $H^1(\rho_{eq}\dd\mb{V})$. The compact embedding $H^1(\rho_{eq}\dd\mb{V})\hookrightarrow\hookrightarrow  L^2(\rho_{eq}\dd\mb{V})$ (see \cite{taylor1996partial}) implies that up to subsequences, there exists $u^*\in L^2(\rho_{eq}\dd\mb{V})$,
	\begin{align*}
	\dfrac{u_n}{\rho_{eq}}\to\dfrac{u^*}{\rho_{eq}}\ \mathrm{in}\ L^2(\rho_{eq}\dd\mb{V}),
	\end{align*}
	or equivalently
	\begin{align*}
	u_n\to u^*\ \mathrm{in}\ L^2(\dd\mb{V}/\rho_{eq}).
	\end{align*}
	Thus $(\lambda-\mathcal{L}^*)^{-1}$ is compact. Thus $1/\lambda\neq 0$ is not an accumulation point of $\sigma((\lambda-\mathcal{L}^*)^{-1})$, hence 0 is not an accumulation point of $\sigma(\mathcal{L}^*)$, but the single principal eigenvalue of $\mathcal{L}^*$, whose eigenvectors are $c\cdot\rho_{eq}$ where $c$ is a constant. So for any $g\in L^2(\dd\mb{V}/\rho_{eq})$, we have
	\begin{align}
	-\int_{O(n)}\dfrac{g}{\rho_{eq}}\mathcal{L}^*g\dd\mb{V} \geq C\int_{O(n)}\dfrac{|g|^2}{\rho_{eq}}\dd\mb{V}
	\end{align}
	for all $g$ satisfying $\int_{O(n)}g\dd\mb{V}=0$. Let $g=\rho_{eq}f$, we have \eqref{eq:Poincare2}.
\end{proof}
\subsection{The case of $n=2$}
If we ask $\mb F,\mb H$ in 
\eqref{eq:diffusionapprox} to satisfy \eqref{eq:condition2}, then the SDE reads as
\begin{align} \label{eq:n=2}
	\mathrm{d}w_{ij} = g_{ij}(\mb{A},\mb{W})\mathrm{d}t+\sqrt{\eta}\cdot N_{ijkl}\circ\mathrm{d}B_{kl}.
\end{align}
Here $\mb{M}(\mb{W})$ is the covariance matrix of $\mb{G}(\mb{x}\mb{x}^T,\mb{W})$ defined in \eqref{eq:cov}. By Lemma \ref{prop:orthogonality}, 
\eqref{eq:n=2} admits the Stiefel manifold as an invariant set. Utilizing this fact, we can reformulate \eqref{eq:n=2}:
\begin{lemma} \label{lmm:n=2reform}
	Suppose $n=2$. Consider \eqref{eq:diffusionapprox} where $\mb{H}$ and $\mb{H}$ satisfy \eqref{eq:condition2}. Then \eqref{eq:diffusionapprox} (or \eqref{eq:n=2}) can be reformulated as 
	\begin{align} \label{eq:reform1}
		\mathrm{d}\mb{W} = \mb{F_1}(\mb{W})\mathrm{d}t +\sqrt{\eta}\cdot c(\mb{W})\mb{W}\circ\mathrm{d}\mb{Z}.
	\end{align} 
	Here $\mb{F_1}$ is defined in \eqref{def:F}, $c(\mb{W})$ is a scalar defined as
	\begin{equation} \label{def:c}
		\begin{aligned}
		c(\mb{W}) &:= \sqrt{c_1(w_{1,1}^4+w_{1,2}^4)+c_2w_{1,1}^2w_{1,2}^2+c_3w_{1,1}w_{1,2}(w_{1,1}w_{2,2}+w_{1,2}w_{2,1})},\\
		c_1 &:=\E[x_1^2x_2^2]-(\E[x_1x_2])^2,\\
		c_2 &:=\E[x_1^4+x_2^4-4x_1^2x_2^2]+2(\E[x_1x_2])^2+2\E[x_1^2]\E[x_2^2]-(\E[x_1^2])^2-(\E[x_2^2])^2,\\\
		c_3 &:=2\E[x_1^3x_2-x_1x_2^3]-\E[x_1^2]\E[x_1x_2]+\E[x_2^2]\E[x_1x_2]
		\end{aligned}
	\end{equation}
	and $\mb{Z}(t)\in\R^{2\times 2}$ is defined as
	\begin{align} \label{eq:Z}
		\mb{Z}(t):=\begin{pmatrix}
		0 & -B(t)\\
		B(t) & 0\\
		\end{pmatrix},
	\end{align} 
	where $B(t)$ is the standard Brownian motion in one dimension. 
\end{lemma}
See Section \ref{sec:appendix} for the proof of this lemma.
Remember that each element in $O(2)$ can be expressed in either of the following forms:
\begin{align} \label{def:O2}
O_1(\theta) &=
\begin{pmatrix}
\cos\theta & \sin\theta \\
-\sin\theta & \cos\theta
\end{pmatrix},\ \theta\in[0,2\pi),\\
O_{-1}(\theta) &=
\begin{pmatrix}
\cos\theta & \sin\theta \\
\sin\theta & -\cos\theta
\end{pmatrix},\ \theta\in[0,2\pi).
\end{align}
Because the orbit of $\mb{W}(t)$ in \eqref{eq:diffusionapprox} is continuous a.s. and the determinant is also a continuous function w.r.t. $\mb{W}$, so if $\mb{W_0}=O_1(\theta)$ for some $\theta$, then $\mb{W}(t)=O_1(\theta(t))$ for some $\theta(t)\in[0,2\pi)$.  
Without loss of generality, we assume $\mb{W_0}=O_1(\theta_0)$ for some $\theta_0\in[0,2\pi)$, thus for any $t\geq 0$, 
\begin{align}
	w_{1,1}(t)=w_{2,2}(t),\ w_{1,2}(t)+w_{2,1}(t)=0,\ w_{1,1}^2(t)+w_{1,2}^2(t)=1.
\end{align}

	To prove the convergence of \eqref{eq:reform1}, we consider the process of $\theta(t)$ instead of $\mb{W}$. We construct the following one dimensional SDE in Ito's sense: 
	\begin{align} \label{eq:theta}
	\mathrm{d}\theta(t)=f(\theta)\mathrm{d}t+\sqrt{\eta}g(\theta)\mathrm{d}B,\ \theta(t)=\theta_0,
	\end{align}
	where $B(t)$ is the standard Brownian motion, $f,g$ are defined as
	\begin{equation} \label{eq:fg}
	\begin{aligned}
	g(\theta) &= -c(\theta)\\
	f(\theta) &= (\E[x_2^2]-\E[x_1^2])\cos\theta\sin\theta+\dfrac{\eta(2c_1(\theta)-c_2(\theta))}{2}(\cos^3\theta\sin\theta-\sin^3\theta\cos\theta)\\
	&+\dfrac{3\eta c_3(\theta)}{2}\cos^2\theta\sin^2\theta.
	\end{aligned}
	\end{equation}
	Here $c,c_1,c_2$ and $c_3$ are the scalar functions defined in \eqref{def:c} and $c(\theta)$ is $c(\mb{W})$ where $\mb{W}$ is replaced by $O_1(\theta)$, i.e.
	\begin{align}
	c(\theta)=c\left(
	\begin{pmatrix}
	\cos\theta & \sin\theta\\
	-\sin\theta & \cos\theta
	\end{pmatrix}
	\right).
	\end{align}
	
Then we can use the solution of \eqref{eq:theta} to represent the solution of \eqref{eq:reform1}.
\begin{lemma} \label{lmm:reform2}
	Consider \eqref{eq:reform1}. Suppose that the initial value $\mb{W}(0)$ satisfies $\mb{W}(0)=O_1(\theta_0)$ (defined in \eqref{def:O2}) for some $\theta_0\in[0,2\pi)$. Suppose that $\theta(t)$ is the solution of \eqref{eq:theta}, then
	\begin{align} \label{eq:w}
		w_{1,1}(t)=w_{2,2}(t)=\cos(\theta(t)), w_{1,2}(t)=\sin(\theta(t)),\ w_{2,1}(t)=-\sin(\theta(t))
	\end{align} 
	solves \eqref{eq:reform1} with the initial value $\mb{W}(0)=O_1(\theta_0)$.
\end{lemma}
See Section \ref{sec:appendix} for the proof of this lemma.

\subsubsection{Convergence analysis}
Let $\T=\mathbb{\R}/{2\pi\mathbb{Z}}$. Now consider \eqref{eq:theta}. We denote the invariant measure of this SDE as $\rho_{\infty}(x)\in\mathcal{P}(\T)$. Then $\rho_{\infty}$ is the stationary solution to the Fokker-Planck equation with the periodic boundary condition, i.e.
\begin{align} \label{eq:FPfortheta}
	\dfrac{\p \rho(x,t)}{\p t}=\mathcal{L}_1\rho(x,t),\ \mathcal{L}_1\rho(x,t):=-\dfrac{\p (f(x)\rho(x,t))}{\p x}+\dfrac{1}{2}\cdot\dfrac{\p^2 (\eta g^2(x)\rho(x,t))}{\p x^2}.
\end{align}
Here $f,g$ are defined in \eqref{eq:fg}
A direct calculation yields
\begin{align} \label{eq:invariantmeasure}
\rho_{\infty}(x)=C\exp\left(\int_0^x\dfrac{2f(s)}{\eta\cdot g^2(s)}\mathrm{d}s\right),\ C=\left(\exp\left(\int_0^{2\pi}\dfrac{2f(s)}{\eta\cdot g^2(s)}\dd s\right)\right)^{-1}.	
\end{align}

According to the proof of Lemma \ref{lmm:reform2} (see Section \ref{sec:appendix}), we know that $g(x)$ is strictly positive on $\T$, thus $\rho_{\infty}(x)>0$ for any $x\in\T$, thus the exponential convergence holds by a Poincare's inequality. Define the weighted $L^2$ space $L^2(\T,\dd x/\rho_{\infty})$ as 
\begin{align} \label{def:weightedL2}
L^2(\T,\dd x/\rho_{\infty}):=\left\{p:\int_{\T}\dfrac{p^2(x)}{\rho_{\infty}(x)}\dd x<\infty\right\}, \langle p,q\rangle_{\dd x/\rho_{\infty}}=\int_{\T}\dfrac{p(x)q(x)}{\rho_{\infty}(x)}\dd x.
\end{align}

\begin{theorem} \label{thm:expconvoftheta}
	Suppose that $\theta(t),\ t\geq 0$ solves \eqref{eq:theta} with the initial value $\theta_0$, which is a random variable with density function $\rho_0\in\mathcal{P}(\T)$. Let $\rho(x,t)\in\mathcal{P}(\T)$ be the law of $\theta(t)$, i.e., $\mathbb{P}(\theta(t)\in B)=\int_{B}\rho(x,t)\dd x$ for any Borel set $B\in \T$. Consider $\rho_{\infty}(x)\in\mathcal{P}(\T)$, i.e. the invariant measure in \eqref{eq:invariantmeasure}. Then there exists a constant $c>0$ only depending on $\eta$ and momentums of $\mb{x}$ such that
	\begin{align} \label{eq:expconvofmeasure}
	\|\rho(\cdot,t)-\rho_{\infty}\|_{L^2(\T,\dd x/\rho_{\infty})}\leq e^{-ct}\|\rho_0-\rho_{\infty}\|_{L^2(\T,\dd x/\rho_{\infty})}.
	\end{align} 
\end{theorem}

See Section \ref{sec:appendix} for proof of Theorem \ref{thm:expconvoftheta}.

\subsubsection{An example} In this section, we explicitly calculated an example here to illustrate our main results. Suppose that $\mb{x}=(x_1,x_2)^T$ where $x_1$ and $x_2$ are independent random variables such that they possess density function 
\begin{align*}
	\rho_1=\dfrac{1}{4}\cdot\mathbbm{1}_{(-2,2)},\  \rho_2=\dfrac{1}{2}\cdot\mathbbm{1}_{(-1,1)},
\end{align*}
i.e. $x_1\sim\mathrm{Uni}(-2,2)$ and $x_2\sim\mathrm{Uni}(-1,1)$. Then the covariance matrix of $\mb{x}$ is 
\begin{align*}
	\mb{A} =
	\begin{pmatrix}
	4/3 & 0\\
	0 & 1/3
	\end{pmatrix}.
\end{align*}
Then by \eqref{def:c}, we have
\begin{align*}
	c_1 = \dfrac{4}{9},\ c_2= \dfrac{8}{45},\ c_3=0,
\end{align*}
and 
\begin{align*}
	c = \sqrt{\dfrac{4}{9}(\cos^4\theta+\sin^4\theta)+\dfrac{8}{45}\cos^2\theta\sin^2\theta} = \sqrt{\dfrac{4}{9}-\dfrac{32}{45}\cos^2\theta\sin^2\theta}.
\end{align*}
Thus $c\geq \sqrt{4/15}$ for any $\theta\in[0,2\pi)$. Then \eqref{eq:theta} reads as
\begin{align*}
	\mathrm{d}\theta = \left(-\cos\theta\sin\theta + \dfrac{16\eta}{45}(\cos^3\theta\sin\theta-\sin^3\theta\cos\theta)\right)\dd t + \sqrt{\dfrac{4}{9}-\dfrac{32}{45}\cos^2\theta\sin^2\theta}\cdot\dd B.
\end{align*}
According to \eqref{eq:invariantmeasure}, the invariant measure is 
\begin{align*}
	\rho_{\infty}(x)  =\exp\left(\dfrac{1}{\eta}\int_0^x\dfrac{-45\cos\theta\sin\theta + 16\eta(\cos^3\theta\sin\theta-\sin^3\theta\cos\theta)}{10-16\cos^2\theta\sin^2\theta}\dd\theta\right),\ x\in[0,2\pi).
\end{align*}
Direct calculation yields
\begin{align} \label{eq:exinv}
	\rho_{\infty}(x)=\exp\left(-\dfrac{15\sqrt{6}}{16\eta}\arctan\dfrac{2\sqrt{6}\sin^2x}{5-4\sin^2x}\right) \sqrt{\dfrac{5}{8\sin^4x-8\sin^2x+5}}, x\in[0,2\pi).
\end{align}
The explicit formula of the invariant measure, i.e. \eqref{eq:exinv}, shows that if $\eta<<1$, then the mass of the invariant measure concentrates around $x=0$ and $x=\pi$, or in terms of the matrix, around $-\mb{I}_2$ and $\mb{I}_2$, which gives the right principal component decomposition.

\section{The case of $p<n$}

All results for the case of $p=n$ can be extended to the case of $p<n$, by being careful on the size of tensors and rewrite the projection operator. 

First, all regularity and stability results for the semigroup still hold. The proof is exactly the same as in Lemma \ref{lmm:stability} and Theorem \ref{thm:semigroup}, by replacing the terminal index $n$ by $p$ for column indices.

Second, the diffusion approximation is now formulated as
\begin{align}\label{eq:pdiffusionapprox}
\dot{\mb W}=\mb G(\mb A,\mb W)+\eta \mathcal{P}_{T_{\mb W}O(n\times p)}\mb F(\mb W)+\sqrt{\eta}\mathcal{P}_{T_{\mb W}O(n\times p)}\dot{\mb Z},\ \mb W(0)=\mb{W_0}\in O(n\times p).
\end{align}
Here $\mathcal{P}_{T_{\mb W}O(n\times p)}$ is the projection onto the tangent space of $O(n\times p)$ at $\mb W$, see \eqref{def:pP}; $\dot{\mb Z}$ is defined as
\begin{align}
\dot{\mb Z}(\mb W)=(\dot{Z}_{ij})_{n\times p},\ \dot{Z}_{ij}(\mb W)=H_{ijkl}(\mb W)\circ\dot{B}_{kl};
\end{align}
$\mb H(\mb W)=(H_{ijkl}(\mb W))_{n\times p\times n\times p}$ is the coefficient tensor of the Brownian motion; $\mb B=(B_{ij})_{n\times p}$ is the standard Brownian motion in $\R^{n\times p}$; The notation '$\circ$' represents that \eqref{eq:pdiffusionapprox} is an SDE in the Stratonovich sense.

The projection operator onto $T_{\mb W}O(n\times p)$ then reads as
\begin{align}\label{def:pP}
	\mathcal{P}_{T_{\mb W}O(n\times p)}\mb M = (\mb{I_n}-\mb W\mb W^T)\mb M+\dfrac{1}{2}\mb W(\mb W^T\mb M-\mb M^T\mb W).
\end{align}
When $p=n$, $\mb W\mb W^T=\mb{I_n}$ so $\mathcal{P}_{T_{\mb W}O(n\times p)}$ is exactly the projection on $O(n)$; when $p=1$, $\mb W^T\mb M=\mb M^T\mb W$ are all scalars, then 
\begin{align}
	\mathcal{P}_{T_{\mb W}O(n\times p)}\mb M = \mathcal{P}_{T_{\mb W}\mathbb{S}^{n-1}}\mb M= (\mb{I_n}-\mb W\mb W^T)\mb M,
\end{align} 
i.e. $\mathcal{P}_{T_{\mb W}O(n\times p)}$ degenerates to the projection onto the unit sphere. 

Using the same technique and method as in Lemma \ref{prop:orthogonality}, one can prove that \eqref{eq:pdiffusionapprox} stays on $O(n\times p)$; moreover, the diffusion approximation is also of first-order as in Theorem \ref{thm:diffusionapprox}. 

For reversibility, the overdamped Langevin is still reversible. All the proof in Section \ref{sec:OLE} holds generally on $O(n\times p)$.

\section*{Acknowledgement}
Jian-Guo Liu was supported in part by the National Science Foundation (NSF) under award DMS-2106988.

\section{Appendix} \label{sec:appendix}
\subsection{Riemannian manifolds and the Stiefel manifold} \label{subsec:stiefel}
We denote the tangent space at $\mb{m}$ on manifold $\mathcal{M}$ as $T_{\mb{m}}\mathcal{M}$, the tangent vector field on $\mathcal{M}$ as $\Gamma(T\mathcal{M})$. The tangent bundle (the disjoint union of the tangent spaces) is denoted as $T\mathcal{M}$.
\begin{definition} (Riemannian manifolds)
	Suppose that $\mathcal{M}$ is a smooth manifold. A Riemannian manifold $(\mathcal{M},g)$ is a smooth mainfold equipped with an inner product $g_{\mb{m}}$ on $T_{\mb{m}}\mathcal{M}$ at each $\mb{m}\in\mathcal{M}$. Moreover, for any tangent vector field $\dot{\mb{x}}$ and $\dot{\mb{y}}$, the function 
	\begin{align}
	\langle\dot{\mb{x}}(\mb{m}),\dot{\mb{y}}(\mb{m})\rangle_{g_{\mb{m}}}:\ \mathcal{M}\to\R
	\end{align}
	is smooth.
\end{definition}

Given a Riemannian metric $g$ on $M$, the gradient of a smooth function $\mathcal{E}$ on $\mathcal{M}$ is defined as 
\begin{definition} \label{def:gradient}
	(the gradient on the Riemannian manifold) A tangent vector field $\nabla_g\mathcal{E}$ on $M$ is called the gradient of $\mathcal{E}$ w.r.t. the metric $g$ if for every tangent vector field $\dot{\mb{x}}$ on $M$,
	\begin{align} \label{eq:grad}
	\langle \mathcal{E}',\ \dot{\mb{x}}\rangle_F = \langle \nabla_g\mathcal{E},\ \dot{\mb{x}}\rangle_g.
	\end{align} 
	Here $\mathcal{E}'$ is the derivative of $\mathcal{E}$, which is a cotangent vector field.
\end{definition}

Now we consider the Stiefel manifold with the Euclidean metric $g_e$, under the global coordinate $\mb{Q}\in\R^{n\times n}$. We first introduce several important properties of $O(n)$. See \cite{helmke2012optimization} for more details.

\begin{lemma} \label{lmm:Stiefel}
	(the Stiefel manifold)
	The Stiefel manifold $O(n)$ is a smooth, compact manifold of dimension $n(n-1)/2$. The tangent space at $\mb{Q}$ is given by
	\begin{align}\label{eq:tangent}
		T_{\mb{Q}}O(n)=\{\mb{Q}\bm{\Omega}\ |\ \bm{\Omega}\in\R^{n\times n},\ \bm{\Omega}+\bm{\Omega}^T=\mb{0}\},
	\end{align}
while the normal space at $\mb{Q}$ is given by 
\begin{align}\label{eq:normal}
T_{\mb{Q}}O(n)^{\perp}=\{\mb{Q}\bm{\Omega}\ |\ \bm{\Omega}\in\R^{n\times n},\ \bm{\Omega}=\bm{\Omega}^T\}.
\end{align}
\end{lemma}

See \cite{helmke2012optimization} for proof of Lemma \ref{lmm:Stiefel}. By Lemma \ref{lmm:Stiefel}, we can prove that for any $\mb{M}\in\R^{n\times n}$, the projection on the tangent spaces and the normal spaces are respectively:
\begin{align}\label{def:P}
	\mathcal{P}_{T_{\mb{Q}}O(n)}\mb{M}:=\dfrac{1}{2}(\mb{M}-\mb{Q}\mb{M}^T\mb{Q}),\ \mathcal{P}_{T_{\mb{Q}}O(n)^{\perp}}\mb{M}:=\dfrac{1}{2}(\mb{M}+\mb{Q}\mb{M}^T\mb{Q}).
\end{align}

\begin{lemma} (Gradient on $(O(n),g_e)$) Suppose that $\varphi(\mb{Q}):O(n)\to\R$ is a restriction of a smooth function (still denoted as $\varphi:\R^{n\times n}\to\R$) on $O(n)$. Then gradient of $\varphi$ w.r.t. $g_e$ at point $\mb{Q}$ is given by
	\begin{align}
		\nabla_{g_e}\varphi:=\mathcal{P}_{T_{\mb{Q}}O(n)}(\nabla\varphi)=\dfrac{1}{2}(\nabla\varphi-\mb{Q}(\nabla\varphi)^T\mb{Q}),
	\end{align}  
	here $
	\nabla\varphi\in\R^{n\times n}$ is the gradient of $\varphi$ in $\R^{n\times n}$, i.e. $(\nabla\varphi)_{ij}=\dfrac{\p\varphi}{\p q_{ij}}$.
\end{lemma}
\begin{proof}
	We just need to prove that for any $\mb{Q}\in O(n)$ and any tangent vector at $\mb{Q}$, \eqref{eq:grad} holds. By Lemma \ref{lmm:Stiefel}, a tangent vector at $\mb{Q}$ can be represented as $\mb{Q}\bm{\Omega}$ where $\bm{\Omega}$ is skew-symmetric. Therefore, for any $\bm{\Omega}$ that is skew-symmetric, we have 
	\begin{align}
		\langle \nabla\varphi, \mb{Q}\bm{\Omega}\rangle_F = \langle \mathcal{P}_{T_{\mb{Q}}O(n)}(\nabla\varphi),\ \mb{Q}\bm{\Omega}\rangle_F + \langle \mathcal{P}_{T_{\mb{Q}}O(n)^{\perp}}(\nabla\varphi),\ \mb{Q}\bm{\Omega}\rangle_F.
	\end{align} 
	Notice that 
	\begin{align*}
		\langle \mathcal{P}_{T_{\mb{Q}}O(n)^{\perp}}(\nabla\varphi),\ \mb{Q}\bm{\Omega}\rangle_F &= \langle \mb{Q}(\mb{Q}^T\nabla\varphi+(\nabla\varphi)^T\mb{Q}),\ \mb{Q}\bm{\Omega}\rangle_F\\
		&= \langle\mb{Q}^T\nabla\varphi+(\nabla\varphi)^T\mb{Q},\ \bm{\Omega}\rangle_F\\
		&=0,
	\end{align*}
	since $\langle \mb{M},\ \mb{N}\rangle_F=0$ if $\mb{M}$ is symmetric while $\mb{N}$ is skew-symmetric. Therefore,
	\begin{align*}
		\langle \nabla\varphi, \mb{Q}\bm{\Omega}\rangle_F = \langle \mathcal{P}_{T_{\mb{Q}}O(n)}(\nabla\varphi),\ \mb{Q}\bm{\Omega}\rangle_F = \langle \mathcal{P}_{T_{\mb{Q}}O(n)}(\nabla\varphi),\ \mb{Q}\bm{\Omega}\rangle_{g_e}.
	\end{align*}
	So $\nabla_{g_e}\varphi=\mathcal{P}_{T_{\mb{Q}}O(n)}(\nabla\varphi)$.
\end{proof}
Under this global coordinate, the divergence on $(O(n), g_e)$ can also be explicitly computed. The $ij$ entry of $\nabla_{g_e}\varphi$ is given by
\begin{align}
	(\nabla_{g_e}\varphi)_{ij}= \dfrac{1}{2}\left(\dfrac{\p\varphi}{\p q_{ij}}-\sum_{k,l}q_{ik}q_{lj}\dfrac{\p\varphi}{\p q_{lk}}\right).
\end{align}
So given a tangent vector field $\mb{H}(\mb{Q})=(h_{ij}(\mb{Q}))$, the divergence of it is defined by
\begin{align}
	\nabla_{g_e}\cdot\mb{H}(\mb{Q}):=\sum_{i,j}(\nabla_{g_e}(h_{ij}(\mb{Q})))_{ij}.
\end{align}
The Laplace-Beltrami operator is then defined as:
\begin{align}
	\Delta_{g_e}\varphi:=\nabla_{g_e}\cdot\nabla_{g_e}\varphi.
\end{align}
Explicit expression of the Laplace-Beltrami operator is given by the following lemma:
\begin{lemma} \label{lmm:LBO}
	(Laplace-Beltrami operator) The Laplace-Beltrami operator on $(O(n),g_e)$ is given by
	\begin{align} \label{eq:LBO}
		\Delta_{g_e}\varphi=\dfrac{1}{2}\left(\sum_{i,j}\dfrac{\p^2\varphi}{\p{q_{ij}^2}}-(n-1)\sum_{i,j}q_{ij}\dfrac{\p\varphi}{\p q_{ij}}-\sum_{i,j,k,l}q_{ik}q_{lj}\dfrac{\p^2\varphi}{\p q_{ij}\p q_{lk}}\right).
	\end{align}
\end{lemma}
\begin{proof}
	By the expression of the divergence and gradient, we have 
	\begin{align*}
	\Delta_{g_e}\varphi=\sum_{i,j}\left(\dfrac{\p h_{ij}}{\p q_{ij}}-\sum_{k,l}q_{ik}q_{lj}\dfrac{\p h_{ij}}{\p q_{lk}}\right),\ h_{ij}=\sum_{i,j}\left(\dfrac{\p\varphi}{\p q_{ij}}-\sum_{k,l}q_{ik}q_{lj}\dfrac{\p\varphi}{\p q_{lk}}\right).
	\end{align*}
	Thus 
	\begin{align*}
		\dfrac{\p h_{ij}}{\p q_{ij}} &= \dfrac{1}{2}\left(\dfrac{\p^2\varphi}{\p q_{ij}^2}-\sum_{k,l}\delta_{kl}q_{lj}\dfrac{\p\varphi}{\p q_{lk}}-\sum_{k,l}\delta_{il}q_{ik}\dfrac{\p\varphi}{\p q_{lk}}-\sum_{k,l}q_{ik}q_{lj}\dfrac{\p^2\varphi}{\p q_{ij}q_{lk}}\right)\\
		&= \dfrac{1}{2}\left(\dfrac{\p^2\varphi}{\p q_{ij}^2}-\sum_{l}q_{lj}\dfrac{\p\varphi}{\p q_{lj}}-\sum_{k}q_{ik}\dfrac{\p\varphi}{\p q_{ik}}-\sum_{k,l}q_{ik}q_{lj}\dfrac{\p^2\varphi}{\p q_{ij}q_{lk}}\right),
	\end{align*}
	and 
	\begin{align*}
		\dfrac{\p h_{ij}}{\p q_{lk}}=\dfrac{1}{2}\left(\dfrac{\p^2\varphi}{\p q_{ij}\p q_{lk}}-\sum_{k',l'}\delta_{il}\delta_{kk'}q_{l'j}\dfrac{\p\varphi}{\p q_{l'k'}}-\sum_{k',l'}\delta_{ll'}\delta_{kj}q_{ik'}\dfrac{\p\varphi}{\p q_{l'k'}}-\sum_{k',l'}q_{ik'}q_{l'j}\dfrac{\p^2\varphi}{\p q_{l'k'}\p q_{lk}}\right).
	\end{align*}
	Substituting the above formulas into the Laplacian operator, we can derive \eqref{eq:LBO}.
\end{proof}

\subsection{Proofs of lemmas and omitted calculations} \label{subsec:proofs}

\subsubsection{Section \ref{sec:diff}}
\begin{proof} [Proof of Lemma \ref{lmm:stability}]
	According to \eqref{def:G} and \eqref{eq:rewriteleading}, direct computation yields
	\begin{equation} \label{eq:Fnorm}
	\begin{aligned}
	\|\mb{W}(k)\|_F^2 &= \|\mb{W}(k-1)\|_F^2+2\eta\cdot\mathrm{tr}(\mb{W}(k-1)^T\mb{A}(k)\mb{W}(k-1))\\
	&-2\eta\cdot\mathrm{tr}(\mb{W}(k-1)^T\mb{W}(k-1)\mb{W}(k-1)^T\mb{A}(k)\mb{W}(k-1))\\
	&+\eta^2\cdot\mathrm{tr}(\mb{W}(k-1)^T\mb{A}(k)^2\mb{W}(k-1))\\ &-2\eta^2\cdot\mathrm{tr}(\mb{W}(k-1)^T\mb{A}(k)\mb{W}(k-1)\mb{W}(k-1)^T\mb{A}(k)\mb{W}(k-1))\\
	&+\eta^2\cdot\mathrm{tr}(\mb{W}(k-1)^T\mb{A}(k)\mb{W}(k-1)\mb{W}(k-1)^T\mb{W}(k-1)\mb{W}(k-1)^T\mb{A}(k)\mb{W}(k-1))\\
	&-2\eta^2\cdot\mathrm{tr}(\mb{W}(k-1)^T\mb{A}(k)\mb{W}(k-1)\mb{W}(k-1)^T\mb{W}(k-1)\mb{\Sigma}(\mb{A}(k),\mb{W}(k-1)))\\
	&+\eta^2\cdot\mathrm{tr}(\mb{\Sigma}(\mb{A}(k),\mb{W}(k-1))^T\mb{W}(k-1)^T\mb{W}(k-1)\mb{\Sigma}(\mb{A}(k),\mb{W}(k-1))).
	\end{aligned}
	\end{equation}
	By definition of $\mb{\Sigma}$ in \eqref{def:G}, we have
	\begin{equation} \label{eq:estsigma}
	\begin{aligned}
	\|\mb{\Sigma(\mb{A}(k),\mb{W}(k-1)}\|_F &= \sqrt{\sum_{i\neq j}(\mb{w}_i(k-1)^T\mb{A}(k)\mb{w}_j(k-1))^2}\\
	&= \sqrt{\sum_{i\neq j}(\mb{w}_i(k-1)^T\mb{x}(k))^2(\mb{w}_j(k-1)^T\mb{x}(k))^2}\\
	&\leq \sum_{i=1}^n(\mb{w}_i(k-1)^T\mb{x}(k))^2\\
	&\leq M^2 \sum_{i=1}^n\|\mb{w}_i(k-1)\|_2^2\\
	&= M^2\|\mb{W}(k-1)\|_F^2.
	\end{aligned}
	\end{equation}
	By the following norm inequality: $\|\mb{M}\|_2\leq \|\mb{M}\|_F\leq \sqrt{n}\|\mb{M}\|_2$, \eqref{ass:L_inf} and \eqref{eq:estsigma}, we derive the following estimates for each term in the above equality in the almost surely sense:
	\begin{equation} \label{eq:est1}
	\begin{aligned}
	\mathrm{tr}(\mb{W}(k-1)^T\mb{A}(k)\mb{W}(k-1))=\|\mb{W}(k-1)^T\mb{x}(k)\|_2^2&\leq \|\mb{W}(k-1)^T\|_2^2\|\mb{x}(k)\|_2^2\\
	&\leq M^2\|\mb{W}(k-1)\|_F^2.
	\end{aligned}
	\end{equation}
	\begin{equation} \label{eq:est2}
	\begin{aligned}
	\mathrm{tr}(\mb{W}(k-1)^T\mb{A}(k)^2\mb{W}(k-1))&=\|\mb{x}(k)\|^2_2\mathrm{tr}(\mb{W}(k-1)^T\mb{A}(k)\mb{W}(k-1))\\
	&\leq M^4\|\mb{W}(k-1)\|_F^2.
	\end{aligned}
	\end{equation}
	\begin{equation} \label{eq:est3}
	\begin{aligned}
	&\ \ \ \   \mathrm{tr}(\mb{W}(k-1)^T\mb{A}(k)\mb{W}(k-1)\mb{W}(k-1)^T\mb{W}(k-1)\mb{W}(k-1)^T\mb{A}(k)\mb{W}(k-1))\\
	&=\mathrm{tr}((\mb{W}(k-1)^T\mb{x}(k))(\mb{x}(k)^T\mb{W}(k-1)\mb{W}(k-1)^T\mb{W}(k-1)\mb{W}(k-1)^T\mb{A}(k)\mb{W}(k-1)))\\
	&\leq \|\mb{W}(k-1)^T\mb{x}(k)\|_2\|\mb{x}(k)^T\mb{W}(k-1)\mb{W}(k-1)^T\mb{W}(k-1)\mb{W}(k-1)^T\mb{A}(k)\mb{W}(k-1)\|_2\\
	&\leq M^2\|\mb{W}(k-1)\|_2^6\|\mb{A}(k)\|_2\\
	&\leq M^4\|\mb{W}(k-1)\|_F^6.
	\end{aligned}
	\end{equation}
	\begin{equation} \label{eq:est4}
	\begin{aligned}
	&\ \ \ \  |\mathrm{tr}(\mb{W}(k-1)^T\mb{A}(k)\mb{W}(k-1)\mb{W}(k-1)^T\mb{W}(k-1)\mb{\Sigma}(\mb{A}(k),\mb{W}(k-1)))|\\
	&=|\mathrm{tr}((\mb{W}(k-1)^T\mb{x}(k))(\mb{x}(k)^T\mb{W}(k-1)\mb{W}(k-1)^T\mb{W}(k-1)\mb{\Sigma}(\mb{A}(k),\mb{W}(k-1))))|\\
	&\leq\|(\mb{W}(k-1)^T\mb{x}(k)\|_2\|\mb{x}(k)^T\mb{W}(k-1)\mb{W}(k-1)^T\mb{W}(k-1)\mb{\Sigma}(\mb{A}(k),\mb{W}(k-1))\|_2\\
	&\leq M^2\|\mb{W}(k-1)\|^4_2\|\mb{\Sigma}(\mb{A}(k),\mb{W}(k-1))\|_2\\
	&\leq M^2\|\mb{W}(k-1)\|_F^4\|\mb{\Sigma}(\mb{A}(k),\mb{W}(k-1))\|_F\\
	&\leq M^4\|\mb{W}(k-1)\|_F^6.
	\end{aligned}
	\end{equation}
	\begin{equation} \label{eq:est5}
	\begin{aligned}
	&\ \ \ \  \mathrm{tr}(\mb{\Sigma}(\mb{A}(k),\mb{W}(k-1))^T\mb{W}(k-1)^T\mb{W}(k-1)\mb{\Sigma}(\mb{A}(k),\mb{W}(k-1)))\\
	&= \|\mb{W}(k-1)\mb{\Sigma}(\mb{A}(k),\mb{W}(k-1))\|_F^2\\
	&\leq n\|\mb{W}(k-1)\mb{\Sigma}(\mb{A}(k),\mb{W}(k-1))\|_2^2\\
	&\leq n\|\mb{W}(k-1)\|_2^2\|\mb{\Sigma}(\mb{A}(k),\mb{W}(k-1))\|_2^2.\\
	&\leq nM^4\|\mb{W}(k-1)\|_F^6.
	\end{aligned}
	\end{equation}
	
	Substituting \eqref{eq:est1} $\sim$ \eqref{eq:est5} into \eqref{eq:Fnorm} yields
	\begin{align}
	\|\mb{W}(k)\|_F^2 \leq (1+2M^2\eta)\|\mb{W}(k-1)\|_F^2 + \eta^2(M^4\|\mb{W}(k-1)\|_F^2+(n+3)M^4\|\mb{W}(k-1)\|_F^6).
	\end{align}
	
	Select $\eta(T)$ as
	\begin{align}
	\eta(T)= \dfrac{1}{M^4(1+(n+3)r^4e^{2T(2M^2+1)})}.
	\end{align} 
	Then we can prove that for any $\eta\leq \eta(T)$, the Markov chain generated by \eqref{eq:leadingorder} which starts from $\mb{W_0}$ satisfies
	\begin{align}
	\|\mb{W}(k)\|_F^2\leq r^2e^{2M^2+1}, k=1,2,...,[T/\eta],\ \mathrm{a.s.}.
	\end{align}
	We prove by induction. Given any $\eta\leq \eta(T)$, if $\|\mb{W}(k-1)\|_F^2\leq r^2e^{T(2M^2+1)}$, then 
	\begin{align*}
	\|\mb{W}(k)\|_F^2 &\leq (1+2M^2\eta)\|\mb{W}(k-1)\|_F^2 + \eta^2(M^4\|\mb{W}(k-1)\|_F^2+(n+3)M^4\|\mb{W}(k-1)\|_F^6)\\
	&\leq (1+2M^2\eta)\|\mb{W}(k-1)\|_F^2 + (\eta M^4(1+(n+3)r^4e^{2T(2M^2+1)}))\cdot\eta\cdot\|\mb{W}(k-1)\|_F^2\\
	&\leq (1+(2M^2+1)\eta)\|\mb{W}(k-1)\|_F^2.
	\end{align*}
	Remember that $\|\mb{W_0}\|_F^2\leq r^2$, so $\|\mb{W_0}\|_F^2\leq r^2e^{2M^2+1}$, hence for any $k=1,2,...,[T/\eta]$,
	\begin{align*}
	\|\mb{W}(k)\|_F^2\leq (1+(2M^2+1)\eta)^kr^2\leq (1+(2M^2+1)\eta)^{T/\eta}r^2\leq r^2e^{T(2M^2+1)}.
	\end{align*}
	Thus taking $C(T)=r^2e^{T(2M^2+1)}$ concludes the proof.
\end{proof}

\begin{proof}[Proof of Lemma \ref{lmm:orth}]
	Because $\mb M=(M_{ijkl})$ is the covariance matrix of $\mb G$, it is positive semidefinite. Therefore, the square root of $\mb M$ is well-defined, which satisfies (i),(ii) and (iii) above. Denote it as $\mb N$. 
	
	When $\mb W\in O(n)$, we have proved that for any symmetric $\mb B\in\R^{n\times n}$, $\mb G(\mb B, \mb W)\in T_{\mb W}O(n)$, so 
	\begin{align*}
	\mathcal{P}_{T_{\mb W}O(n)}\mb G(\mb A-\mb x\mb x^T,\mb W) = \mb G(\mb A-\mb x\mb x^T,\mb W),
	\end{align*} 
	i.e. $P_{ijkl}g_{kl}=g_{ij}$ by Lemma \ref{lmm:projection}. Therefore, by linearity of expectation and symmetry
	\begin{align*}
	M_{ijkl} &= \E[g_{ij}(\mb A-\mb x\mb x^T,\mb W)g_{kl}(\mb A-\mb x\mb x^T,\mb W)]\\
	&=\E[P_{ijrs}g_{rs}(\mb A-\mb x\mb x^T,\mb W)g_{xy}(\mb A-\mb x\mb x^T,\mb W)P_{klxy}]\\
	&=P_{ijrs}\E[g_{rs}(\mb A-\mb x\mb x^T,\mb W)g_{xy}(\mb A-\mb x\mb x^T,\mb W)]P_{xykl}\\
	&=P_{ijrs}M_{rsxy}P_{xykl},
	\end{align*}
	or in matrix notation, $\mb M=\mb P\mb M\mb P$. Thus by (iii) in Lemma \ref{lmm:projection}, we have
	\begin{align*}
	P_{ijrs}M_{rskl} = P_{ijrs}P_{rsuv}M_{uvxy}P_{xykl} = P_{ijrs}P_{uvrs}M_{uvxy}P_{xykl} = P_{ijuv}M_{uvxy}P_{xykl},
	\end{align*}
	or $\mb{PM}=\mb{PMP}$. Similarly, 
	\begin{align*}
	M_{ijrs}P_{rskl} = P_{ijuv}M_{uvxy}P_{xyrs}P_{rskl} = P_{ijuv}M_{uvxy}P_{xykl},
	\end{align*}
	or $\mb{MP}=\mb{PMP}$. Thus 
	\begin{align*}
	M_{ijrs}P_{rskl} = P_{ijrs}M_{rskl}.
	\end{align*}
	So $\mb P$ and $\mb M$ are commutative. Because $\mb P$ is symmetric, so $\mb P$ is also commutative with the square root of $\mb M$. Thus 
	\begin{align*}
	N_{ijrs}P_{rskl} = P_{ijrs}N_{rskl}.
	\end{align*} 
	Thus 
	\begin{align*}
	P_{ijrs}N_{rsxy}P_{xyuv}N_{uvkl} = P_{ijrs}N_{rsxy}N_{xyuv}P_{uvkl} = P_{ijrs}M_{rsuv}P_{uvkl} = M_{ijkl},
	\end{align*}
	or $\mb{PNPN} = \mb{PNNP} =\mb{PMP} = \mb M$. So $PN$ is also the square root of $\mb M$, by uniqueness,
	\begin{align*}
	P_{ijrs}N_{rskl}=N_{ijkl},
	\end{align*}
	i.e. $\mb{PN}=\mb N$.
\end{proof}

\begin{proof} [Calculations in Theorem \ref{thm:diffusionapprox}]
	Direct calculation yields
	\begin{equation}\label{eq:taylorS}
	\begin{aligned}
	S\vp(\mb{W}) &= \E\left[\vp(\mb{W}+\eta\mb{G}(\mb{x}\mb{x}^T,\mb{W}))\right]\\
	&= \vp(\mb{W}) +\left(g_{ij}(\mb A,\mb{W})\dfrac{\p\vp}{\p w_{ij}}\right)\eta+\left(\dfrac{1}{2}\E\left[g_{ij}(\mb x\mb x^T,\mb W)g_{kl}(\mb x\mb x^T,\mb W)\right]\dfrac{\p^2 \varphi}{\p w_{ij}\p w_{kl}}\right)\eta^2+O(\eta^3).
	\end{aligned}
	\end{equation}
	
	Meanwhile,
	\begin{equation*}
	\begin{aligned}
	e^{\eta\mathcal{L}}\varphi(\mb W) = \varphi(\mb W)+\eta\mathcal{L}\varphi(\mb W) +\dfrac{1}{2}\eta^2\mathcal{L}^2\varphi(\mb W)+O(\eta^3). 
	\end{aligned}
	\end{equation*}
	
	By \eqref{eq:L}, we have
	\begin{align}
	\begin{aligned}
	\mathcal{L}^2\varphi &=\mathcal{L}\left[g_{ij}(\mb A,\mb W)\dfrac{\p\vp}{\p w_{ij}}+O(\eta)\right]\\
	&=g_{kl}(\mb A,\mb W)\dfrac{\p }{\p w_{kl}}\left(g_{ij}(\mb A,\mb W)\dfrac{\p\vp}{\p w_{ij}}\right)+O(\eta)\\
	&=\left(g_{kl}(\mb A,\mb W)\dfrac{\p g_{ij}(\mb A,\mb W)}{\p w_{kl}}\right)\dfrac{\p\vp}{\p w_{ij}}+g_{ij}(\mb A,\mb W)g_{kl}(\mb A,\mb W)\dfrac{\p^2\vp}{\p w_{ij}\p w_{kl}}+O(\eta).
	\end{aligned}
	\end{align}
	
	Thus 
	\begin{equation}\label{eq:mid}
		\begin{aligned} 
		e^{\eta\mathcal{L}}\varphi(\mb W)&=\varphi(\mb W) + \eta\left[\left(g_{ij}(\mb A,\mb W)+\eta P_{ijkl}f_{kl} + \dfrac{\eta}{2}J_{ij}\right)\dfrac{\p \varphi}{\p w_{ij}}+\dfrac{\eta}{2}K_{ijrs}K_{klrs}\dfrac{\p^2\vp}{\p w_{ij}w_{kl}}\right]\\ 
		&+\dfrac{\eta^2}{2}\left[\left(g_{kl}(\mb A,\mb W)\dfrac{\p g_{ij}(\mb A,\mb W)}{\p w_{kl}}\right)\dfrac{\p\vp}{\p w_{ij}}+g_{ij}(\mb A,\mb W)g_{kl}(\mb A,\mb W)\dfrac{\p^2\vp}{\p w_{ij}\p w_{kl}}\right]+O(\eta^3).
		\end{aligned}
	\end{equation}

	Recast \eqref{eq:mid} in the order of $\eta$, we have
	\begin{equation}\label{eq:taylorL}
		\begin{aligned}
		e^{\eta\mathcal{L}}\varphi(\mb W)&=\varphi(\mb W) +\left(g_{ij}(\mb A,\mb{W})\dfrac{\p\vp}{\p w_{ij}}\right)\eta +\\
		&\left[\left(P_{ijkl}f_{kl}+\dfrac{1}{2}J_{ij}+\dfrac{1}{2}g_{kl}(\mb A,\mb W)\dfrac{\p g_{ij}(\mb A,\mb W)}{\p w_{kl}}\right)\dfrac{\p \varphi}{\p w_{ij}}\right]\eta^2\\
		&+\left[\left(K_{ijrs}K_{klrs}+g_{ij}(\mb A,\mb W)g_{kl}(\mb A,\mb W)\right)\dfrac{\p^2\vp}{\p w_{ij}\p w_{kl}}\right]\eta^2+O(\eta^3).
		\end{aligned}
	\end{equation}
	
	Comparing the $\eta^2$ term in \eqref{eq:taylorS} and \eqref{eq:taylorL}, we have
	\begin{align*}
	\dfrac{\p\vp}{\p w_{ij}}&: P_{ijkl}f_{kl}+\dfrac{1}{2}J_{ij}+\dfrac{1}{2}g_{kl}(\mb A,\mb W)\dfrac{\p g_{ij}(\mb A,\mb W)}{\p w_{kl}}=0;\\
	\dfrac{\p^2\vp}{\p w_{ij}\p w_{kl}}&:
	K_{ijrs}K_{klrs}+g_{kl}(\mb A,\mb W)g_{ij}(\mb A,\mb W) = \E\left[g_{ij}(\mb x\mb x^T,\mb W)g_{kl}(\mb x\mb x^T,\mb W)\right].
	\end{align*}
	Thus $\mb F=(f_{ij})_{n\times n}$ and $\mb H=(H_{ijkl})_{n\times n\times n\times n}$ should satisfy
	\begin{equation*} 
	\begin{aligned}
	P_{ijkl}f_{kl}&=-\dfrac{1}{2}J_{ij}-\dfrac{1}{2}g_{kl}(\mb A,\mb W)\dfrac{\p g_{ij}(\mb A,\mb W)}{\p w_{kl}},\\
	P_{ijuv}H_{uvrs}P_{klxy}H_{xyrs} &= M_{ijkl}=\E[g_{ij}(\mb A-\mb x\mb x^T,\mb W)g_{kl}(\mb A-\mb x\mb x^T,\mb W)],
	\end{aligned}
	\end{equation*}
	which is \eqref{eq:second}

\end{proof}

\subsubsection{Section \ref{sec:sde}}

\begin{proof} [Proof of Lemma \ref{lmm:n=2reform}]
	We first compute $\mb{H(\mb{W})}$. By Proposition \ref{prop:orthogonality}, we know that $\mb{W}\mb{W}^T=\mb{I_2}$. Thus
	\begin{align*}
		\mb{G(\mb{x}\mb{x}^T-A,\mb{W})}&=(\mb{x}\mb{x}^T-\mb{A})\mb{W}-\mb{W}\mb{W}^T(\mb{x}\mb{x}^T-\mb{A})\mb{W}+\mb{W}\mb{\Sigma}(\mb{x}\mb{x}^T-\mb{A},\mb{W})\\
		&=\mb{W}\mb{\Sigma}(\mb{x}\mb{x}^T-\mb{A},\mb{W}).
	\end{align*}
	Denote $\mb{B}=\mb{x}\mb{x}-\mb{A}$. Since $n=2$, we have 
	\begin{align*}
		\mb{\Sigma}(\mb{B},\mb{W})=
		\begin{pmatrix}
		0 & -\mb{w_1}\cdot\mb{B}\mb{w_2}\\
		\mb{w_1}\cdot\mb{B}\mb{w_2} & 0\\
		\end{pmatrix}.
	\end{align*}
	Denote $b=\mb{w_1}\cdot\mb{B}\mb{w_2}$, then 
	\begin{align*}
		\mb{G(\mb{x}\mb{x}^T-A,\mb{W})} = \mb{W}\begin{pmatrix}
		0 & -b\\
		b & 0\\
		\end{pmatrix} 
		= 
		\begin{pmatrix}
		w_{1,2}b & -w_{1,1}b\\
		w_{2,2}b & -w_{2,1}b\\
		\end{pmatrix}.
	\end{align*}
	Thus 
	\begin{align*}
		\mb{vec}(\mb{G}(\mb{B},\mb{W})) = (w_{1,2}b,\ -w_{1,1}b,\ w_{2,2}b,\ -w_{2,1}b)^T.
	\end{align*}
	Denote $\mb{u}=(w_{1,2},\ -w_{1,1},\ w_{2,2},\ -w_{2,1})^T$. Then the covariance matrix defined in \eqref{def:covariance} is 
	\begin{equation}
		\begin{aligned}
		\mb{M(W)} &= \E[\mb{vec}(\mb{G}(\mb{B},\mb{W}))\mb{vec}(\mb{G}(\mb{B},\mb{W}))^T]\\
		&=\E[b^2]\mb{u}\mb{u}^T.
		\end{aligned}
	\end{equation}
	Therefore, 
	\begin{align*}
	\mb{H(W)} = \sqrt{\mb{M(\mb{W})}} = \dfrac{\sqrt{\E[b^2]}}{\|\mb{u}\|_2}\mb{u}\mb{u}^T.
	\end{align*}
	Remember that $\mb{W}\in O(2)$, thus $\|\mb{u}\|_2=\sqrt{2}$. So we have
	\begin{align} \label{eq:H}
	\mb{H(W)}= \dfrac{\sqrt{\E[b^2]}}{\sqrt{2}}\mb{u}\mb{u}^T.
	\end{align}
	
	Substituting \eqref{eq:H} into \eqref{eq:reform1}, we have 
	\begin{align}
		\mathrm{d}\mb{vec}(\mb{W}) = \mb{vec}(\mb{G}(\mb{A},\mb{W}))\mathrm{d}t + \dfrac{\sqrt{\eta\E[b^2]}}{\sqrt{2}}\mb{u}\mb{u}^T\circ\mathrm{d}\mb{B}.
	\end{align}
	Moreover, because $\|\mb{u}\|_2=\sqrt{2}$ and $\mb{B}(t)$ is the standard Brownian motion in $\R^4$, thus $\mb{u}^T\mathrm{d}\mb{B}$ has the same law with $\sqrt{2}\mathrm{d}B$ where $B(t)$ is the standard Brownian motion in one dimension. Thus \eqref{eq:reform1} can also be reformulated as
	\begin{align}
		\mathrm{d}\mb{vec}(\mb{W}) = \mb{vec}(\mb{G}(\mb{A},\mb{W}))\mathrm{d}t + \sqrt{\eta\E[b^2]}\mb{u}\circ\mathrm{d}B(t).
	\end{align}
	Moreover, because $\mb{W}(t)\in O(2)$, thus $\mb{G}(\mb{A},\mb{W})=\mb{F}(\mb{W})$ where $\mb{F}$ is defined in \eqref{def:F}. Thus \eqref{eq:reform1} can be directly transformed in the form of matrices:
	\begin{align}
	\mathrm{d}\mb{W} = \mb{F(W)}\mathrm{d}t + \sqrt{\eta\E[b^2]}\cdot\mb{Q}\circ\mathrm{d}\mb{Z},
	\end{align}
	where $\mb{Z}$ is defined in \eqref{eq:Z}.
	
	Now we compute $\E[b^2]$. Direct computation yields
	\begin{align*}
		b = \mb{w_1}^T\mb{B}\mb{w_2} = b_{1,1}w_{1,1}w_{1,2}+b_{1,2}(w_{1,1}w_{2,2}+w_{2,1}w_{1,2})+b_{2,2}w_{2,1}w_{2,2}.
	\end{align*}
	Thus 
	\begin{align*}
		b^2 &= b_{1,1}^2w_{1,1}^2w_{1,2}^2+b_{1,2}^2(w_{1,1}w_{2,2}+w_{1,2}w_{2,1})^2+b_{2,2}^2w_{2,1}^2w_{2,2}^2+2b_{1,1}b_{2,2}w_{1,1}w_{1,2}w_{2,1}w_{2,2}\\
		&+2b_{1,1}b_{1,2}w_{1,1}w_{2,2}w_{1,2}+2b_{1,1}b_{1,2}w_{1,1}w_{1,2}^2w_{2,1}+2b_{1,2}b_{2,2}w_{1,1}w_{2,1}w_{2,2}^2+2b_{1,2}b_{2,2}w_{1,2}w_{2,1}^2w_{2,2}.
	\end{align*}
	Meanwhile, we know $b_{i,j}=x_ix_j-\E[x_ix_j]$ for $1\leq i,j\leq 2$, thus for any $i,j,i',j'=1,2$, we have
	\begin{align*}
		\E[b_{i,j}b_{i',j'}]=\E[x_ix_jx_{i'}x_{j'}]-\E[x_ix_j]\E[x_{i'}x_{j'}].
	\end{align*} 
	Substituting this into expression of $b^2$, we have
	\begin{align*}
		\E[b^2]&=\mathrm{var}(x_1^2)w_{1,1}^2w_{1,2}^2+\mathrm{var}(x_1x_2)(w_{1,1}^2w_{2,2}^2+w_{2,1}^2w_{1,2}^2+2w_{1,1}w_{1,2}w_{2,1}w_{2,2})+\mathrm{var}(x_2^2)w_{2,1}^2w_{2,2}^2\\
		&+2(\E[x_1^2x_2^2]-\E[x_1^2]\E[x_2^2])w_{1,1}w_{1,2}w_{2,1}w_{2,2}+2(\E[x_1^3x_2]-\E[x_1^2]\E[x_1x_2])w_{1,1}w_{1,2}(w_{1,1}w_{2,2}+w_{1,2}w_{2,1})\\
		&+2(\E[x_1x_2^3]-\E[x_2^2]\E[x_1x_2])w_{2,1}w_{2,2}(w_{1,1}w_{2,2}+w_{1,2}w_{2,1}).
	\end{align*}
	Using $w_{1,1}^2=w_{2,2}^2,w_{1,2}^2=w_{2,1}^2$, we then derive
	\begin{align} \label{eq:b}
		\E[b^2]=c_1(\mb{W})(w_{1,1}^4+w_{1,2}^4)+c_2(\mb{W})w_{1,1}^2w_{1,2}^2+c_3(\mb{W})w_{1,1}w_{1,2}(w_{1,1}w_{2,2}+w_{1,2}w_{2,1}),
	\end{align}
	where $c_1,c_2$ and $c_3$ is defined in \eqref{def:c}. Thus by taking $c(\mb{W})=\sqrt{\E[b^2]}$, we derive \eqref{eq:reform1} which concludes the proof.
\end{proof}

\begin{proof} [Proof of Lemma \ref{lmm:reform2}]
	We first write \eqref{eq:reform1} in It\'o's sense. Because we have assumed that $|\mb{W}(t)|=1$, thus we only need to consider the equation for $w_{1,1}$ and $w_{1,2}$. In the It\'o sense, we have
	\begin{align*}
		\mathrm{d}w_{1,1} &=[(\mb{F(\mb{W})})_{1,1}+h_1(\mb{W})]\mathrm{d}t+\sqrt{\eta}\cdot c(\mb{W})w_{1,2}\mathrm{d}B\\ 
		\mathrm{d}w_{1,2} &=[(\mb{F(\mb{W})})_{1,2}+h_2(\mb{W})]\mathrm{d}t-\sqrt{\eta}\cdot c(\mb{W})w_{1,1}\mathrm{d}B,
	\end{align*}
	where $h_1$ and $h_2$ are defined in \ref{eq:correction}. Denote $g_1=\sqrt{\eta}c(\mb{W})w_{1,2}, g_2=-\sqrt{\eta}c(\mb{W})w_{1,1}$, then $h_1$ and $h_2$ are computed as
	\begin{equation}
		\begin{aligned}
		h_1(\mb{W})&=\dfrac{1}{2}\left(g_1\dfrac{\p g_1}{\p w_{1,1}}+g_2\dfrac{\p g_1}{\p w_{1,2}}\right)=\dfrac{\eta}{2}\left(\dfrac{1}{2}w_{1,2}^2\dfrac{\p c(\mb{W})^2}{\p w_{1,1}}-\dfrac{1}{2}w_{1,1}w_{1,2}\dfrac{\p c(\mb{W}^2)}{\p w_{1,2}}-w_{1,1}c(\mb{W})^2\right)\\
		h_2(\mb{W})&=\dfrac{1}{2}\left(g_1\dfrac{\p g_2}{\p w_{1,1}}+g_2\dfrac{\p g_2}{\p w_{1,2}}\right)=\dfrac{\eta}{2}\left(\dfrac{1}{2}w_{1,1}^2\dfrac{\p c(\mb{W})^2}{\p w_{1,2}}-\dfrac{1}{2}w_{1,1}w_{1,2}\dfrac{\p c(\mb{W}^2)}{\p w_{1,1}}-w_{1,2}c(\mb{W})^2\right).
		\end{aligned}
	\end{equation}
	Here we used the chain rule: $c\dfrac{\p c}{\p w}=\dfrac{1}{2}\dfrac{\p c^2}{\p w}$. Substituting the above equation into the SDE in It\'o's sense yields
	\begin{equation} \label{eq:help2}
		\begin{aligned} 
		\mathrm{d}w_{1,1} &= \left((\E[x_2^2]-\E[x_1^2])w_{1,1}w_{1,2}+\dfrac{\eta}{4}w_{1,2}^2\dfrac{\p c^2}{\p w_{1,1}}-\dfrac{\eta}{4}w_{1,1}w_{1,2}\dfrac{\p c^2}{\p w_{1,2}}-\dfrac{\eta}{2}w_{1,1}c(\mb{W})^2\right)\mathrm{d}t\\
		&+\sqrt{\eta}c(\mb{W})w_{1,2}\mathrm{d}B.
		\end{aligned}
	\end{equation}
	
	Now consider the process of $\theta$. Suppose that $\theta(t)$ satisfies the following SDE in Ito's sense:
	\begin{align*}
		\mathrm{d}\theta = f(\theta)\mathrm{d}t + \sqrt{\eta}\cdot g(\theta)\mathrm{d}B.
	\end{align*}
	Then Ito's isometry yields 
	\begin{align*}
		\mathrm{d}\cos\theta &= -\sin\theta\cdot\mathrm{d}\theta -\dfrac{\eta g^2(\theta)}{2}\mathrm{d}t\\
		&= \left(-\sin\theta\cdot f(\theta)-\dfrac{\eta\cos\theta}{2}g^2(\theta)\right)\mathrm{d}t-\sqrt{\eta}\cdot\sin(\theta)g(\theta)\mathrm{d}B.
	\end{align*}
	Replacing $w_{1,1}$ by $\cos\theta$, $w_{1,2}$ by $\sin\theta$ in \eqref{eq:help2} and comparing the coefficients of it with the above equation yields
	\begin{equation}
	\begin{aligned}
	g(\theta) &= -c(\theta),\\
	f(\theta) &= (\E[x_2^2]-\E[x_1^2])\sin\theta\cos\theta+\eta\cdot\dfrac{2c_1(\theta)-c_2(\theta)}{2}(\cos^3\theta\sin\theta-\cos\theta\sin^3\theta)\\
	&+\eta\cdot\dfrac{3c_3(\theta)}{2}\cos^2\theta\sin^2\theta.
	\end{aligned}
	\end{equation} 
	Here we used \eqref{eq:b} since $c^2(\mb{W})=\E[b^2]$. This is exactly the SDE in \eqref{eq:theta}.
\end{proof}

\begin{proof}[Proof of Theorem \ref{thm:expconvoftheta}]
	In the following proof, $C$ is just a general constant that may vary among equations. Because $\rho(x,t),\ t\geq 0$ is the law of $\theta(t)$, so $\rho$ solves \eqref{eq:FPfortheta} on $\T\times\R^+$. Thus by the periodic boundary condition,
	\begin{align*}
	\dfrac{\dd}{\dd t}\int_{\T}\rho(x,t)\dd x= \dfrac{1}{2}\int_{\T}\left[-\dfrac{\p (f(x)\rho(x,t))}{\p x}+\dfrac{1}{2}\cdot\dfrac{\p^2 (\eta g^2(x)\rho(x,t))}{\p x^2}\right]\dd x=0.
	\end{align*}
	So $\int_{\T}\rho(x,t)\dd x=1,\ t\geq 0$.
	
	Now consider the Fokker-Planck operator $\mathcal{L}_1^*$ which is self-adjoint in $L^2(\T,\dd x/\rho_{\infty})$:
	\begin{align*}
	\mathcal{L}^*_1:D(\mathcal{L}^*_1)\subset L^2(\T,\dd x/\rho_{\infty})\to L^2(\T,\dd x/\rho_{\infty}),\ \mathcal{L}_1^*p:=-\dfrac{\dd}{\dd x}\left(\rho_{\infty}(x)\dfrac{\dd}{\dd x}\left(\dfrac{p(x)}{\rho_{\infty}(x)}\right)\right). 
	\end{align*} 
	Here $D(\mathcal{L}^*_1)=H^2(\T,\dd x/\mu)$. A direct calculation yields
	\begin{align*}
	\langle \mathcal{L}1^*p,q\rangle_{\dd x/\rho_{\infty}}=\int_{\T}\rho_{\infty}(x)\dfrac{\dd}{\dd x}\left(\dfrac{p(x)}{\rho_{\infty}(x)}\right)\dfrac{\dd}{\dd x}\left(\dfrac{q(x)}{\rho_{\infty}(x)}\right)\dd x,
	\end{align*}
	thus $\mathcal{L}_1^*$ is semi-positive definite, and 
	\begin{align*}
	\mathcal{L}_1^*p=0\iff p(x)=c\rho_{\infty}(x).
	\end{align*}
	So $0$ is the simple principle eigenvalue of $\mathcal{L}_1^*$ with $\rho_{\infty}(x)$ as the eigenvector.
	
	Moreover, $0$ is isolated. For any $\lambda>0$, we prove that $(\lambda+\mathcal{L}_1^*)^{-1}$ is compact. Let $\{g_n\}_{n=1}^{\infty}\in L^2(\T,\dd x/\rho_{\infty})$ be a bounded sequence, with 
	\begin{align*}
	(\lambda+\mathcal{L}^*)u_n=g_n, n=1,2,....
	\end{align*}
	To prove that $(\lambda+\mathcal{L}^*)^{-1}$ is compact, we just need to prove that there exists a subsequence of $\{u_n\}_{n=1}^{\infty}$ which is Cauchy in $L^2(\T,\dd x/\rho_{\infty})$. Because $\mathcal{L}^*$ is semi-positive definite, so $(\lambda+\mathcal{L}^*)$ is bounded, thus $\{u_n\}_{n=1}^{\infty}$ is bounded in $L^2(\T,\dd x/\rho_{\infty})$. By the Cauchy-Schwatz inequality, we have
	\begin{align}
	\langle\mathcal{L}_1^*u_n,u_n\rangle_{\dd x/\rho_{\infty}} = \langle u_n,g_n\rangle_{\dd x/\rho_{\infty}} -\lambda \|u_n\|_{L^2(\T,\dd x/\rho_{\infty})}^2\leq C.
	\end{align}
	Here $C$ is a constant. Thus 
	\begin{align*}
	\|u_n/\rho_{\infty}\|_{H^1(\T,\rho_{\infty}\dd x)}^2 &=\int_\T\rho_{\infty}(x)\left(\dfrac{\dd}{\dd x}\dfrac{u_n(x)}{\rho_{\infty}(x)}\right)^2\dd x + \int_{\T}\rho_{\infty}(x)\left(\dfrac{u_n(x)}{\rho_{\infty}(x)}\right)^2\dd x \\
	&= \|u_n\|_{L^2(\T,\dd x/\rho_{\infty})}+\langle \mathcal{L}_1^*u_n,u_n\rangle_{\dd x/\rho_{\infty}}\leq C.
	\end{align*}
	So $u_n/\rho_{\infty}$ is bounded in $H^1(\T,\rho_{\infty}\dd x)$. By the compact embedding $H^1(\T,\rho_{\infty}\dd x)\subset\subset L^2(\T,\rho_{\infty}\dd x)$, we know that there exists a subsequence of $u_n$ (still denoted as $u_n$) such that 
	\begin{align*}
	\dfrac{u_n}{\rho_{\infty}}\to \dfrac{u^*}{\rho_{\infty}}\ \mathrm{in}\ L^2(\T,\rho_{\infty}\dd x),
	\end{align*} 
	or equivalently
	\begin{align*}
	u_n\to u^*\ \mathrm{in}\ L^2(\T,\dd x/\rho_{\infty}).
	\end{align*}
	So $(\lambda+\mathcal{L}^*)^{-1}$ is a compact operator. Thus the spectrum of $(\lambda+\mathcal{L}^*)^{-1}$ only admits 0 as an accumulation point. So $0$ is an isolated point in the spectrum of $\mathcal{L}^*$. Thus for any $p\in L^2(\T,\dd x/\rho_{\infty})$ such that $\int_{\T}p(x)\dd x=0$, we have the following Poincare's inequality: there exists a constant $c>0$ such that
	\begin{align} \label{eq:Poincare}
	\int_\T\rho_{\infty}(x)\left(\dfrac{\dd}{\dd x}\dfrac{p(x)}{\rho_{\infty}(x)}\right)^2\dd x=\langle \mathcal{L}_1^*p,p\rangle_{L^2(\T,\dd x/\rho_{\infty})}\geq c\|p\|_{L^2(\T,\dd x/\rho_{\infty})}^2=c\int_\T\dfrac{|p(x)|^2}{\rho_{\infty}(x)}\dd x.
	\end{align}
	Multiplying $\rho(x,t)/\rho_{\infty}(x)-1$ on both sides of \eqref{eq:FPfortheta} and substituting $p=\rho(x,t)-\rho_{\infty}(x)$ in \eqref{eq:Poincare} yields
	\begin{align}
	\dfrac{1}{2}\dfrac{\dd }{\dd t}\int_{\T}\dfrac{(\rho(x,t)-\rho_{\infty}(x))^2}{\rho_{\infty}(x)}\dd x = -\int_{\T}\rho_{\infty}(x)\left(\dfrac{\dd }{\dd x}\dfrac{\rho(x,t)}{\rho_{\infty}(x)}\right)^2\dd x\leq -c\int_\T\dfrac{|\rho(x,t)-\rho_{\infty}(x)|^2}{\rho_{\infty}(x)}\dd x.
	\end{align} 
	Thus by Gronwall's inequality, \eqref{eq:expconvofmeasure} holds.
\end{proof}

\bibliographystyle{plain}
\bibliography{PCAODE}
\end{document}